\documentclass{amsart}
\usepackage{amsmath,amsfonts,amsthm,amssymb}
\usepackage{subcaption}
\usepackage{tikz}
\usetikzlibrary{arrows,decorations.pathmorphing,decorations.markings,backgrounds,positioning,fit,calc,shapes,patterns}
\newtheorem{thm}{Theorem}[subsection]

\newtheorem{lemma}[thm]{Lemma}
\newtheorem{conj}[thm]{Conjecture}
\newtheorem{coro}[thm]{Corollary}
\newtheorem{prop}[thm]{Proposition}
\theoremstyle{definition}
\newtheorem{defn}[thm]{Definition}
\theoremstyle{remark}
\newtheorem{rem}[thm]{Remark}
\newtheorem{ex}[thm]{Example}
\numberwithin{equation}{section}

\def\ZZ{\mathbb{Z}}

\def\NN{\mathbb{N}}

\newcommand{\gmat}[2][ccccccccccccccccccccccccccccccccc]{\left[\begin{array}{#1} #2\\ \end{array}\right]}

\colorlet{gpurple}{red!35!blue}

\tikzstyle{mutable}=[inner sep=1mm,circle,draw,minimum size=2mm]
\tikzstyle{frozen}=[inner sep=.9mm,rectangle,draw]
\tikzstyle{dot} = [fill=black!25,inner sep=0.5mm,circle,draw,minimum size=1mm]
\tikzstyle{marked}=[inner sep=0.5mm,circle,draw,blue!75!black,fill=blue!50]
\tikzstyle{oriented}=[draw=blue!50,thick,decoration={markings,mark=at position 0.52 with {\arrow{>}}},postaction={decorate}]
\tikzstyle{wall}=[draw=blue!50,thick]

\title{The existence of a maximal green sequence is not invariant under quiver mutation}  

\author{Greg Muller \today}

\begin{document}

\begin{abstract}
This note demonstrates that the quiver in Figure \ref{fig: 223} does not admit a maximal green sequence.  Since this quiver is mutation-equivalent to a quiver which does admit a maximal green sequence, this provides a counterexample to the conjecture that the existence of maximal green sequences is invariant under quiver mutation.  The proof uses the `scattering diagrams' of Gross-Hacking-Keel-Kontsevich to show that a maximal green sequence for a quiver determines a maximal green sequence for any induced subquiver.
\end{abstract}

\maketitle

\def\Q{\mathsf{Q}}

%\section{Statement of results}
%
%
%%A \emph{maximal green sequence} for a quiver $\Q$ without loops or 2-cycles is a sequence of vertices of $\Q$, which describe a sequence of \emph{mutations} of
%
%\begin{thm}
%If a quiver $\Q$ admits a maximal green sequence, then any induced subquiver $\Q^\dagger$ admits a maximal green sequence.
%\end{thm}
%
%\begin{prop}
%The quiver $\Q_{2,2,3}$ (see Figure \ref{fig: 223}) does not admit a maximal green sequence, but it is mutation equivalent to a quiver which does.
%\end{prop}

\begin{figure}[h!t]	
	\begin{tikzpicture}
	\begin{scope}
		\path[use as bounding box] (-1,-.5) rectangle  (1,1);
		\node[mutable,inner sep=.15cm] (1) at (90:1) {};
		\node[mutable,inner sep=.15cm] (2) at (-30:1) {};
		\node[mutable,inner sep=.15cm] (3) at (210:1) {};
		\draw[-angle 90,relative, out=15,in=165] (1) to (2);
		\draw[-angle 90,relative, out=-15,in=-165] (1) to (2);
		\draw[-angle 90,relative, out=15,in=165] (2) to (3);
		\draw[-angle 90,relative, out=-15,in=-165] (2) to (3);
		\draw[-angle 90,relative, out=25,in=155] (3) to (1);
		\draw[-angle 90,relative] (3) to (1);
		\draw[-angle 90,relative, out=-25,in=-155] (3) to (1);
	\end{scope}
	\end{tikzpicture}
\caption{The counterexample}
\label{fig: 223}
\end{figure}
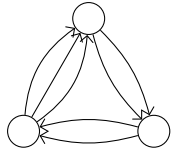

%Quiver mutation is an operation which may be iteratively performed on a quiver $\Q$.

\section{Quivers and maximal green sequences}

%This section introduces the maximal green sequences on a quiver.

\subsection{Quiver mutation}

A \textbf{quiver} is a finite directed graph without loops or 2-cycles.  A quiver $\Q$ may be \textbf{mutated} at a vertex $k$ to produce a new quiver $\mu_k(\Q)$ in three steps.

\noindent \begin{minipage}[l]{3.2in}
	\vspace{.1in}
	\begin{enumerate}
		\item For each pair of arrows $i\rightarrow k \rightarrow j$ through the vertex $k$, add an arrow $i\rightarrow j$.
		\item Reverse the orientation of every arrow incident to $k$.
		\item Cancel any directed 2-cycles in pairs.
	\end{enumerate}
\end{minipage}
\hspace{.3in}
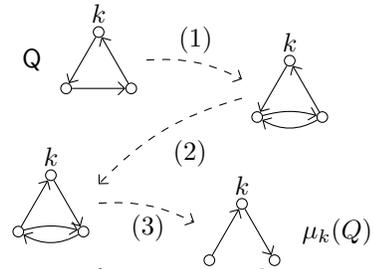
\begin{minipage}[l]{1.5in}
	\begin{tikzpicture}
	\node at (-.1in,.05in) {$\Q$};
	\begin{scope}[xshift=.25in,yshift=0in,scale=.5]
		\node[dot, fill=white] (x1) at (90:1) {};
		\node[dot, fill=white] (x2) at (210:1) {};
		\node[dot, fill=white] (x3) at (-30:1) {};
		\node[above] at (x1) {$k$};
		\draw[-angle 90] (x1) to (x2);
		\draw[-angle 90] (x2) to (x3);
		\draw[-angle 90] (x3) to (x1);
	\end{scope}
	\draw[dashed, -angle 90,relative, out=15,in=165] (.5in,.05in) to node[above] {(1)} (1.0in,-.05in);
	\begin{scope}[xshift=1.25in,yshift=-.15in,scale=.5]
		\node[dot, fill=white] (x1) at (90:1) {};
		\node[dot, fill=white] (x2) at (210:1) {};
		\node[dot, fill=white] (x3) at (-30:1) {};
		\node[above] at (x1) {$k$};
		\draw[-angle 90] (x1) to (x2);
		\draw[-angle 90] (x3) to (x1);
		\draw[-angle 90,relative, out=15,in=165] (x2) to (x3);
		\draw[angle 90-,relative, out=-30,in=-150] (x2) to (x3);
	\end{scope}
	\draw[dashed, -angle 90,relative, out=-15,in=195] (1.0in,-.15in) to node[below right] {(2)} (.25in,-.6in);
	\begin{scope}[xshift=0in,yshift=-.75in,scale=.5]
		\node[dot, fill=white] (x1) at (90:1) {};
		\node[dot, fill=white] (x2) at (210:1) {};
		\node[dot, fill=white] (x3) at (-30:1) {};
		\node[above] at (x1) {$k$};
		\draw[-angle 90] (x2) to (x1);
		\draw[-angle 90] (x1) to (x3);
		\draw[-angle 90,relative, out=15,in=165] (x2) to (x3);
		\draw[angle 90-,relative, out=-30,in=-150] (x2) to (x3);
	\end{scope}
	\draw[dashed, -angle 90,relative, out=15,in=165] (.25in,-.7in) to node[below] {(3)} (.75in,-.8in);
	\begin{scope}[xshift=1.0in,yshift=-.9in,scale=.5]
		\node[dot, fill=white] (x1) at (90:1) {};
		\node[dot, fill=white] (x2) at (210:1) {};
		\node[dot, fill=white] (x3) at (-30:1) {};
		\node[above] at (x1) {$k$};
		\draw[-angle 90] (x2) to (x1);
		\draw[-angle 90] (x1) to (x3);
	\end{scope}
	\node at (1.5in,-.85in) {$\mu_k(Q)$};
	\end{tikzpicture}
\end{minipage}

Two quivers are \textbf{mutation equivalent} if there is a sequence of mutations taking one to the other.  Mutation equivalence is an equivalence relation, because mutating at the same vertex twice in a row returns to the original quiver.

\begin{rem}
Quiver mutation was introduced in \cite{FZ02} to encode the combinatorial part of \emph{seed mutation} in the theory of cluster algebras.  %The same operation was independently introduced in physics to encode \emph{Seiberg dualities} between quiver gauge theories \cite{???}.
\end{rem}

\subsection{$g$-vectors}

In this section, we use the Sign Coherence Theorem for quivers to give an elementary definition of $g$-vectors, via \emph{$g$-seeds}.
A \textbf{$g$-seed} is a quiver $\Q$, together with a basis of the lattice $\ZZ^N$ indexed by the vertices of $\Q$.  The basis element $g_k$ associated to a vertex $k$ is called its \textbf{$g$-vector}.

\begin{ex}
Let $\Q$ be a quiver, and index the vertices of $\Q$ by the numbers $1,2,...,N$.  A $g$-seed may be made from $\Q$, where the $g$-vector of the $i$th vertex is the standard basis element $e_i:=(0,0,...,1,...,0)$.  A $g$-seed of this form is called an \textbf{initial $g$-seed}.
\end{ex}

\begin{figure}[h!t]	
	\begin{tikzpicture}
	\begin{scope}
		\path[use as bounding box] (-1,-.5) rectangle  (1,1);
		\node[mutable] (1) at (90:1) {$1$};
		\node[mutable] (2) at (-30:1) {$2$};
		\node[mutable] (3) at (210:1) {$3$};
		\draw[-angle 90,relative, out=15,in=165] (1) to (2);
		\draw[-angle 90,relative, out=-15,in=-165] (1) to (2);
		\draw[-angle 90,relative, out=15,in=165] (2) to (3);
		\draw[-angle 90,relative, out=-15,in=-165] (2) to (3);
		\draw[-angle 90,relative, out=25,in=155] (3) to (1);
		\draw[-angle 90,relative] (3) to (1);
		\draw[-angle 90,relative, out=-25,in=-155] (3) to (1);
	\end{scope}
	\draw[-angle 90,double distance=.5mm] (.75in,.25) to (1.25in,.25);
	\begin{scope}[xshift=2in]
		\path[use as bounding box] (-1,-.5) rectangle  (1,1);
		\node[rounded rectangle,draw,fill=green!10] (1) at (90:1) {$1,0,0$};
		\node[rounded rectangle,draw,fill=green!10] (2) at (-25:1) {$0,1,0$};
		\node[rounded rectangle,draw,fill=green!10] (3) at (205:1) {$0,0,1$};
		\draw[-angle 90,relative, out=15,in=165] (1) to (2);
		\draw[-angle 90,relative, out=-15,in=-165] (1) to (2);
		\draw[-angle 90,relative, out=15,in=165] (2) to (3);
		\draw[-angle 90,relative, out=-15,in=-165] (2) to (3);
		\draw[-angle 90,relative, out=25,in=155] (3) to (1);
		\draw[-angle 90,relative] (3) to (1);
		\draw[-angle 90,relative, out=-25,in=-155] (3) to (1);
	\end{scope}
	\end{tikzpicture}
\caption{Constructing an initial $g$-seed from an indexed quiver}
\label{fig: initial	}
\end{figure}
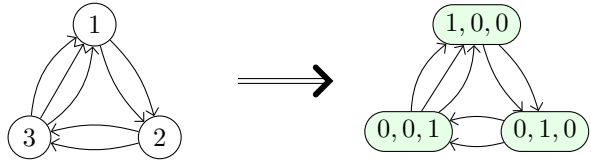

Given a $g$-seed, any element $v\in \ZZ^N$ may be uniquely expressed as\footnote{The coefficient $c_i(v)$ is linear on $\ZZ^N$, and given by the dot product with the \emph{$c$-vector} of vertex $i$.}
\[ v = \sum_{\text{vertices }k\text{ of }\Q} c_k(v) g_k,\;\;\; c_k(v)\in \ZZ \]
A vertex $k$ in a $g$-seed is \textbf{green} if, for all $v\in\NN^N\subset \ZZ^N$, the coefficient $c_k(v)\geq0$.  Similarly, a vertex $k$ in a $g$-seed is \textbf{red} if, for all $v\in \NN^N\subset \ZZ^N$, the coefficient $c_k(v)\leq 0$.  As an example, every vertex in an initial $g$-seed is green.

A $g$-seed may be \textbf{mutated} at a vertex $k$ which is green or red.  The underlying quiver mutates as in the previous section, and the $g$-vectors stay the same except at vertex $k$, which changes as follows.
\[ \mu_k(g_k) := \left\{\begin{array}{cc}
\displaystyle -g_k +\sum_{\forall \text{ arrows } j\rightarrow k} g_j & \text{if $k$ is green} \\
\displaystyle -g_k +\sum_{\forall \text{ arrows } j\leftarrow k} g_j & \text{if $k$ is red} \\
\end{array}\right\} \]
Mutation always turns a green vertex into a red vertex and vice versa, so mutating twice in a row at the same vertex returns to the original $g$-seed.

%In general, a vertex in a $g$-seed may be neither red or green, and hence mutation is not defined at that vertex.  However, w
%We will only consider $g$-seeds that obtained from mutations of an initial $g$-seed, and the following theorem says that every resulting vertex will be either green or red.

\begin{thm}[Sign coherence, \cite{DWZ10}]
After any sequence of mutations of an initial $g$-seed, every vertex in the resulting $g$-seed is either green or red.
\end{thm}
\noindent Hence, a $g$-seed which is mutation equivalent to an initial $g$-seed may be mutated at an arbitrary sequence of vertices.

\begin{rem}
Sign coherence was originally conjectured in \cite{FZ07} in the larger generality of \emph{skew-symmetrizable matrices}.  The case of quivers was first proven in \cite{DWZ10}, and a categorical proof appeared in \cite{Pla11}.  The larger generality of skew-symmetrizable matrices was proven in \cite{GHKK}.
\end{rem}

\begin{rem}
The $g$-vectors of cluster variables with respect to a fixed initial seed were introduced in \cite{FZ07}, as the degrees in the universal grading of the cluster algebra with principal coefficients.  We have used the Sign Coherence Theorem to simplify the recursive identities (6.12) and (6.13) in \cite{FZ07}; specifically, we have removed the explicit dependence on a choice of initial seed.
\end{rem}

%If a $g$-seed $S$ is mutation equivalent to an initial seed, then the $g$-vector of any frozen vertex must be a standard basis element $e_i$ for some $i$.  The \textbf{deletion} of a set of frozen vertices in $S$ is a new $g$-seed $S^\dagger$, obtained by deleting the selected frozen vertices from the quiver, and projecting the $g$-vectors along the map $\ZZ^N\rightarrow \ZZ^{N'}$ which kills the $g$-vectors of the selected frozen vertices.  It is straightforward to check that deletion commutes with mutation.

\subsection{Maximal green sequences}

Starting with an initial $g$-seed, mutation equivalent $g$-seeds will have all their vertices colored green or red.  The following theorem states that monochromatic $g$-seeds in this equivalence class must be unique.% (up to the appropriate notion of `$g$-seed isomorphism').

\begin{thm}\label{thm: greenred}\cite[Prop. 2.10]{BDP12}
Let $(\Q,\{g_k\})$ be a $g$-seed which is mutation equivalent to an initial $g$-seed $(\Q_{in},\{e_k\})$.
\begin{enumerate}
	\item If every vertex of $(\Q,\{g_k\})$ is green, then there is a quiver isomorphism $f:\Q_{in} \rightarrow \Q$ such that $g_{f(k)}=e_k$.
	\item If every vertex of $(\Q,\{g_k\})$ is red, then there is a quiver isomorphism $f:\Q_{in} \rightarrow \Q$ such that $g_{f(k)}=-e_k$.
\end{enumerate}
\end{thm}

While an all-green $g$-seed always exists in such a mutation equivalence class (the initial seed), there may be no all-red $g$-seed. Two examples of such quivers are given in Figure \ref{fig: nog2r} (for proofs, see \cite[Prop. 2.21]{BDP12} and \cite[Theorem 1]{Sev14}).

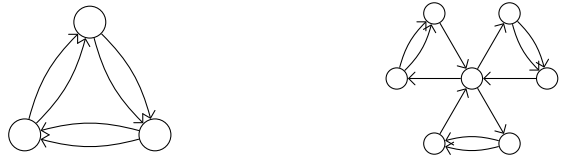
\begin{figure}[h!t]
\begin{tikzpicture}
	\begin{scope}[yshift=-.25cm]
		%\path[use as bounding box] (-1,-1) rectangle  (1,1);
		\node[mutable,inner sep=.15cm] (1) at (90:1) {};
		\node[mutable,inner sep=.15cm] (2) at (-30:1) {};
		\node[mutable,inner sep=.15cm] (3) at (210:1) {};
		\draw[-angle 90,relative, out=15,in=165] (1) to (2);
		\draw[-angle 90,relative, out=15,in=165] (2) to (3);
		\draw[-angle 90,relative, out=15,in=165] (3) to (1);
		\draw[-angle 90,relative, out=-15,in=195] (1) to (2);
		\draw[-angle 90,relative, out=-15,in=195] (2) to (3);
		\draw[-angle 90,relative, out=-15,in=195] (3) to (1);
%		\draw[-angle 90,relative, out=15,in=165] (1) to node[anchor=240] {$2$} (2);
%		\draw[-angle 90,relative, out=15,in=165] (2) to node[below] {$2$} (3);
%		\draw[-angle 90,relative, out=15,in=165] (3) to node[anchor=-60] {$2$} (1);
	\end{scope}
	\begin{scope}[xshift=2in]
		\path[use as bounding box] (-1,-1) rectangle  (1,1);
		\node[mutable,inner sep=.1cm] (0) at (0,0) {};
		\node[mutable,inner sep=.1cm] (1) at (60:1) {};
		\node[mutable,inner sep=.1cm] (2) at (0:1) {};
		\node[mutable,inner sep=.1cm] (3) at (-60:1) {};
		\node[mutable,inner sep=.1cm] (4) at (-120:1) {};
		\node[mutable,inner sep=.1cm] (5) at (-180:1) {};
		\node[mutable,inner sep=.1cm] (6) at (120:1) {};
%		\draw[-angle 90,relative, out=15,in=165] (1) to node[anchor=240] {$2$} (2);
%		\draw[-angle 90,relative, out=15,in=165] (3) to node[below] {$2$} (4);
%		\draw[-angle 90,relative, out=15,in=165] (5) to node[anchor=-60] {$2$} (6);
		\draw[-angle 90,relative, out=15,in=165] (1) to (2);
		\draw[-angle 90,relative, out=15,in=165] (3) to (4);
		\draw[-angle 90,relative, out=15,in=165] (5) to (6);
		\draw[-angle 90,relative, out=-15,in=-165] (1) to (2);
		\draw[-angle 90,relative, out=-15,in=-165] (3) to (4);
		\draw[-angle 90,relative, out=-15,in=-165] (5) to (6);
		\draw[-angle 90] (0) to (1);
		\draw[-angle 90] (2) to (0);
		\draw[-angle 90] (0) to (3);
		\draw[-angle 90] (4) to (0);
		\draw[-angle 90] (0) to (5);
		\draw[-angle 90] (6) to (0);
	\end{scope}
\end{tikzpicture}
\caption{Initial $g$-seeds with these quivers cannot be mutated to an all-red $g$-seed.}
\label{fig: nog2r}
\end{figure}

% (see Section \ref{section: greentored}).
%
%\begin{ex}
%An 
%
%There is no all-red $g$-seed for the Markov.
%\end{ex}

Motivated by applications to non-commutative Donaldson-Thomas theory, Keller introduced the idea of a \emph{maximal green sequence} in \cite{Kel11c}.
\begin{defn}
A \textbf{maximal green sequence} for a quiver $\Q$ is a sequence of mutations starting at an initial $g$-seed $(\Q,\{e_k\})$, such that
\begin{itemize}
	\item each mutation is at a green vertex, and
	\item every vertex in the final $g$-seed is red.
\end{itemize}
\end{defn}

\begin{figure}[h!t]	
	\begin{tikzpicture}
	\begin{scope}
		\path[use as bounding box] (-1,-.5) rectangle  (1,1);
		\node[rounded rectangle,draw,fill=green!10] (1) at (90:1) {$1,0,0$};
		\node[rounded rectangle,draw,fill=green!10] (2) at (-30:1) {$0,1,0$};
		\node[rounded rectangle,draw,fill=green!10] (3) at (210:1) {$0,0,1$};
		\draw[-angle 90,relative, out=15,in=165] (2) to (3);
		\draw[-angle 90,relative, out=-15,in=-165] (2) to (3);
		\draw[-angle 90,relative, out=15,in=165] (3) to (1);
	\end{scope}
	\draw[black!50!green,-angle 90,dashed,out=-65,in=165] (.5,-1) to (1.5,-1.75);
	\begin{scope}[xshift=3cm,yshift=-2cm]
		\path[use as bounding box] (-1,-.5) rectangle  (1,1);
		\node[rounded rectangle,draw,fill=green!10] (1) at (90:1) {$1,0,0$};
		\node[rounded rectangle,draw,fill=red!10] (2) at (-30:1) {$0,-1,0$};
		\node[rounded rectangle,draw,fill=green!10] (3) at (210:1) {$0,0,1$};
		\draw[angle 90-,relative, out=15,in=165] (2) to (3);
		\draw[angle 90-,relative, out=-15,in=-165] (2) to (3);
		\draw[-angle 90,relative, out=15,in=165] (3) to (1);
	\end{scope}
	\draw[black!50!green,-angle 90,dashed,out=65,in=-165] (3.5,-.5) to (4.5,.25);
	\begin{scope}[xshift=6cm]
		\path[use as bounding box] (-1,-.5) rectangle  (1,1);
		\node[rounded rectangle,draw,fill=green!10] (1) at (90:1) {$1,0,0$};
		\node[rounded rectangle,draw,fill=red!10] (2) at (-30:1) {$0,-1,0$};
		\node[rounded rectangle,draw,fill=red!10] (3) at (210:1) {$0,0,-1$};
		\draw[-angle 90,relative, out=15,in=165] (2) to (3);
		\draw[-angle 90,relative, out=-15,in=-165] (2) to (3);
		\draw[angle 90-,relative, out=15,in=165] (3) to (1);
	\end{scope}
	\draw[black!50!green,-angle 90,dashed,out=-65,in=165] (6.5,-1) to (7.5,-1.75);
	\begin{scope}[xshift=9cm,yshift=-2cm]
		\path[use as bounding box] (-1,-.5) rectangle  (1,1);
		\node[rounded rectangle,draw,fill=red!10] (1) at (90:1) {$-1,0,0$};
		\node[rounded rectangle,draw,fill=red!10] (2) at (-30:1) {$0,-1,0$};
		\node[rounded rectangle,draw,fill=red!10] (3) at (210:1) {$0,0,-1$};
		\draw[-angle 90,relative, out=15,in=165] (2) to (3);
		\draw[-angle 90,relative, out=-15,in=-165] (2) to (3);
		\draw[-angle 90,relative, out=15,in=165] (3) to (1);
	\end{scope}
	\end{tikzpicture}
\caption{An example of a maximal green sequence}
\label{fig: maxgreen}
\end{figure}
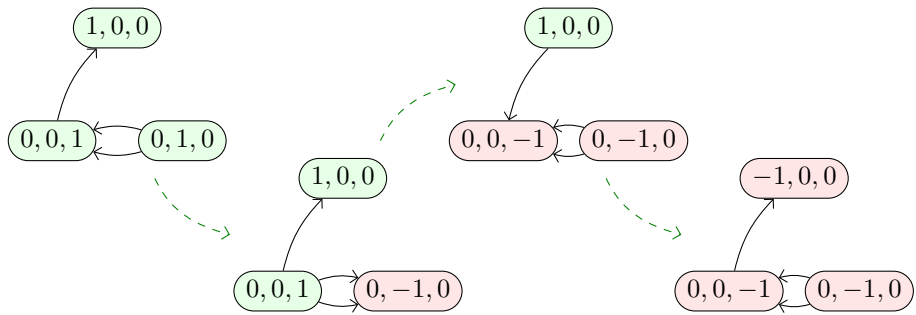

The original purpose of maximal green sequences was to produce identities between products of quantum dilogarithms \cite{Kel11c}.  However, they have been observed in quivers associated to several families of well-behaved \emph{cluster algebras}, such as
\begin{itemize}
	\item acyclic quivers \cite[Lemma 2.20]{BDP12},
	\item quivers from many marked surfaces with boundary \cite{ACCERV13},\cite{Buc14}, and
	\item certain quivers defining the coordinate ring of a double Bruhat cell \cite{YakMGS}.
\end{itemize}
Consequently, it has been conjectured that the existence of a maximal green sequence is equivalent to other good properties of a cluster algebra; for example, that the cluster algebra is equal to its upper cluster algebra.  However, since the cluster algebra only depends on the mutation equivalence class of a quiver, these conjectural equivalences implicitly require the following conjecture.

\begin{conj}\label{conj: MGS}
If a maximal green sequence exists for a quiver $\Q$, then a maximal green sequence exists for any mutation equivalent quiver $\Q'$.
\end{conj}
One purpose of this note is to provide a counterexample to this conjecture.

\subsection{Induced subquivers}

Given a subset $V$ of the vertices of a quiver $\Q$, the \textbf{induced subquiver} $\Q^\dagger$ is the quiver with vertex set $V$ and arrow set consisting of arrows in $\Q$ between pairs of vertices in $V$.  

The key tool in this note is the following theorem.

\begin{thm}\label{thm: induced}
If a quiver $\Q$ admits a maximal green sequence, then any induced subquiver $\Q^\dagger$ admits a maximal green sequence.  
If a maximal green sequence for $\Q$ begins with a sequence of mutations on vertices in $\Q^\dagger$, then there is a maximal green sequence for $\Q^\dagger$ which begins with the same sequence of mutations.
\end{thm}

The proof of theorem uses the theory of \emph{scattering diagrams}, developed in \cite{GPS10}, \cite{KS13}, \cite{GHK11} and connected to cluster algebras in \cite{GHKK}.  Since the relevant theory is elaborate and self-contained, the proof of Theorem \ref{thm: induced} will be deferred to Section \ref{section: proofs}.

\section{Existence and non-existence of maximal green sequences}

This section collects several results about maximal green sequences, culminating in the promised counterexample.

\subsection{Acyclic quivers}

A quiver is \textbf{acyclic} if there are no directed cycles.  Acyclic quivers are the most straightforward class of quivers with a maximal green sequence.

\begin{prop}\cite[Lemma 2.20]{BDP12}\label{thm: acyclic}
An acyclic quiver $\Q$ admits a maximal green sequence.
\end{prop}
\noindent In fact, a maximal green sequence can always be given by iteratively mutating at source vertices which have not yet been mutated.

\subsection{Quivers with $2$ vertices}

Let $b$ be a positive integer, and let $\Q_{b}$ denote the quiver on two vertices with $b$-many arrows from $1$ to $2$ (Figure \ref{fig: b}).

\begin{figure}[h!t]	
	\begin{tikzpicture}
	\begin{scope}
%		\path[use as bounding box] (-1,-1) rectangle  (1,1);
%		\node (label) at (-2,0) {$\Q_{a,b,c}=$};
		\node[mutable] (1) at (-1,0) {$1$};
		\node[mutable] (2) at (1,0) {$2$};
		\draw[-angle 90] (1) to node[above] {$b$} (2);
	\end{scope}
	\end{tikzpicture}
\caption{The quiver $\Q_{b}$}
\label{fig: b}
\end{figure}
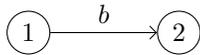

Every quiver of the form $\Q_b$ is acyclic, and thus it admits a maximal green sequence by iteratively mutating at sources.  The following lemma asserts that this is the only maximal green sequence.

\begin{lemma}\label{lemma: rank2}
If $b\geq2$, then the only maximal green sequence for $\Q_b$ is mutation at vertex $1$ and then vertex $2$.
\end{lemma}

\begin{proof}
It is easy to check that mutation at $1$ and then $2$ is a maximal green sequence, and that it is the only maximal green sequence beginning with mutation at $1$.

Consider the mutation of an initial $g$-seed with quiver $\Q_b$ at $2$ (Figure \ref{fig: rank2mutate}).  
%Make $\Q_a$ into an initial $g$-seed and mutate it at $2$ (as in Figure \ref{fig: rank2mutate}).
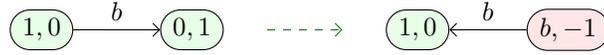
\begin{figure}[h!t]	
	\begin{tikzpicture}
	\begin{scope}
		\node[rounded rectangle,draw,fill=green!10] (1) at (-1,0) {$1,0$};
		\node[rounded rectangle,draw,fill=green!10] (2) at (1,0) {$0,1$};
		\draw[-angle 90] (1) to node[above] {$b$} (2);
	\end{scope}
	\draw[black!50!green,-angle 90,dashed] (2,0) to (3,0);
	\begin{scope}[xshift=5cm]
		\node[rounded rectangle,draw,fill=green!10] (1) at (-1,0) {$1,0$};
		\node[rounded rectangle,draw,fill=red!10] (2) at (1,0) {$b,-1$};
		\draw[angle 90-] (1) to node[above] {$b$} (2);
	\end{scope}
	\end{tikzpicture}
\caption{Mutating an initial seed at vertex $2$}
\label{fig: rank2mutate}
\end{figure}

Define polynomials $F_i$ of $b$ for all $i\in \NN$ recursively, by\footnote{Up to reindexing and a change of variables, these are the \emph{Chebyshev polynomials of the second kind.}}
\[ F_0 = 0,\;\;\; F_1 = 1\]
\[ F_{i+1} = bF_i -F_{i-1} \]

We now make the following claim.  For $i\geq1$, after a sequence of $i$-many green mutation starting at $2$, the resulting $g$-seed is isomorphic to the $g$-seed in Figure \ref{fig: rank2seed}.
\begin{figure}[h!t]
\begin{tikzpicture}
\begin{scope}
	\node[rounded rectangle,draw,fill=green!10] (1) at (-1,0) {$F_i,-F_{i-1}$};
	\node[rounded rectangle,draw,fill=red!10] (2) at (2,0) {$F_{i+1},-F_i$};
	\draw[angle 90-] (1) to node[above] {$b$} (2);
\end{scope}
\end{tikzpicture}
\caption{The $g$-seed after $i$ green mutations starting at $2$}
\label{fig: rank2seed}
\end{figure}
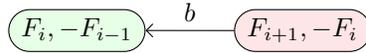

Figure \ref{fig: rank2mutate} establishes the claim for $i=1$.  Assume the claim holds for $i$-many mutations starting at $2$.  The $g$-seed in Figure \ref{fig: rank2seed} has a unique green vertex, so a sequence of $(i+1)$-many green mutations starting at $2$ must come from mutation the this seed at the green vertex.  The resulting seed, shown in Figure \ref{fig: rank2mutate2}, is of the desired form.  
\begin{figure}[h!t]	
	\begin{tikzpicture}
	\begin{scope}
	\node[rounded rectangle,draw,fill=green!10] (1) at (-1.5,0) {$F_i,-F_{i-1}$};
	\node[rounded rectangle,draw,fill=red!10] (2) at (1.5,0) {$F_{i+1},-F_i$};
	\draw[angle 90-] (1) to node[above] {$b$} (2);
	\end{scope}
	\draw[black!50!green,-angle 90,dashed] (1.5,-.666) to (3.5,-1.333);
	\begin{scope}[xshift=5cm, yshift=-2cm]
	\node[rounded rectangle,draw,fill=red!10] (1) at (-2,0) {$bF_{i+1}-F_i,-bF_i+F_{i-1}$};
	\node[rounded rectangle,draw,fill=green!10] (2) at (2,0) {$F_{i+1},-F_i$};
	\draw[-angle 90] (1) to node[above] {$b$} (2);
	\end{scope}
	\end{tikzpicture}
\caption{Mutating the $g$-seed in Figure \ref{fig: rank2seed} at the green vertex}
\label{fig: rank2mutate2}
\end{figure}
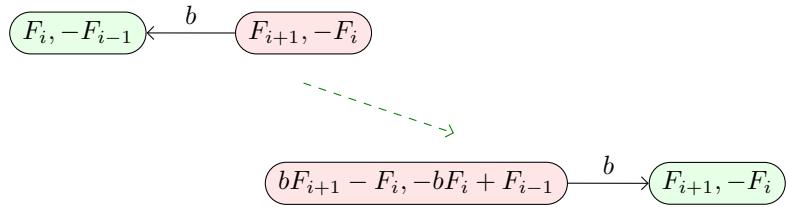

By induction, the claim holds for all $i$.  In particular, no sequence of green mutations starting at $2$ produces a $g$-seed with all vertices red, and so there are no maximal green sequences starting at $2$.
%\noindent Let 
%\[ \C_{in},\C_1,\C_2... \C_n\]
%be the cones of $g$-vectors in the maximal green sequence, so that
%\[ \C_{in}= \{ (a_1,a_2) \in \RR^2 \mid a_1,a_2\geq 0\} \]
%\[ \C_1 = \{ (a_1,a_2)\in \RR^2 \mid a_2\leq 0, \; a_1+ba_2\geq 0\} \]
%
\end{proof}

\subsection{Quivers with $3$ vertices}

Let $a,b,c$ be non-negative integers, and let $\Q_{a,b,c}$ denote the quiver on $3$ vertices with $a$-many arrows from $1$ to $2$, $b$-many arrows from $2$ to $3$, and $c$-many arrows from $3$ to $1$ (Figure \ref{fig: abc}).

\begin{figure}[h!t]	
	\begin{tikzpicture}
	\begin{scope}
		\path[use as bounding box] (-1,-1) rectangle  (1,1);
		\node[mutable] (1) at (90:1) {$1$};
		\node[mutable] (2) at (-30:1) {$2$};
		\node[mutable] (3) at (210:1) {$3$};
		\draw[-angle 90,relative, out=15,in=165] (1) to node[above right] {$a$} (2);
		\draw[-angle 90,relative, out=15,in=165] (2) to node[below] {$b$} (3);
		\draw[-angle 90,relative, out=15,in=165] (3) to node[above left] {$c$} (1);
	\end{scope}
	\end{tikzpicture}
\caption{The quiver $\Q_{a,b,c}$}
\label{fig: abc}

\end{figure}
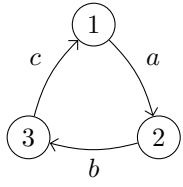

\begin{thm}
If $a,b,c\geq2$, then the quiver $\Q_{a,b,c}$ does not admit a maximal green sequence.
\end{thm}
\begin{proof}
%Assume there is a maximal green sequence for $\Q_{a,b,c}$.  
Assume there is a maximal green sequence which starts by mutating at the vertex $2$.  By Theorem \ref{thm: induced}, this determines a maximal green sequence for the induced subquiver on the vertices $\{1,2\}$ which starts by mutating at vertex $2$, which is impossible by Lemma \ref{lemma: rank2}.  By symmetric arguments, a maximal green sequence cannot start with $1$ or $3$.
%
%(Figure \ref{fig: mutate1}).  
%
%\begin{figure}[h!t]	
%	\begin{tikzpicture}
%	\begin{scope}
%		\path[use as bounding box] (-1,-1) rectangle  (1,1);
%		\node[rounded rectangle,draw,fill=green!10] (1) at (90:1) {$1,0,0$};
%		\node[rounded rectangle,draw,fill=green!10] (2) at (-30:1) {$0,1,0$};
%		\node[rounded rectangle,draw,fill=green!10] (3) at (210:1) {$0,0,1$};
%		\draw[-angle 90,relative, out=15,in=165] (1) to node[above right] {$a$} (2);
%		\draw[-angle 90,relative, out=15,in=165] (2) to node[below] {$b$} (3);
%		\draw[-angle 90,relative, out=15,in=165] (3) to node[above left] {$c$} (1);
%	\end{scope}
%	\draw[-angle 90] (2,0) to node[above] {$\mu_1$} (3,0);
%	\begin{scope}[xshift=5cm]
%		\path[use as bounding box] (-1,-1) rectangle  (1,1);
%		\node[rounded rectangle,draw,fill=red!10] (1) at (90:1) {$-1,2,0$};
%		\node[rounded rectangle,draw,fill=green!10] (2) at (-30:1) {$0,1,0$};
%		\node[rounded rectangle,draw,fill=green!10] (3) at (210:1) {$0,0,1$};
%		\draw[angle 90-,relative, out=15,in=165] (1) to node[above right] {$a$} (2);
%		\draw[angle 90-,relative, out=15,in=165] (2) to node[below] {$ac-b$} (3);
%		\draw[angle 90-,relative, out=15,in=165] (3) to node[above left] {$c$} (1);
%	\end{scope}
%	\end{tikzpicture}
%\caption{Mutation of the initial $g$-seed at $1$ (if $b>ac$, the }
%\label{fig: abc}
%\end{figure}
\end{proof}

However, there are many quivers of the form $\Q_{a,b,c}$ which are mutation equivalent to an acyclic quiver.

\begin{lemma}\cite[Theorem 1.1]{BBH11}\label{thm: nonacyclicrank3}
The quiver $\Q_{a,b,c}$ is mutation-equivalent to an acyclic quiver if and only if
\[abc-a^2-b^2-c^2+4< 0\text{ or } \min(a,b,c)<2\]
\end{lemma}

Combining the last two results provides many counter-examples to Conjecture \ref{conj: MGS}.
\begin{coro}
Let $a,b,c\geq2$ with $abc-a^2-b^2-c^2+4< 0$.  Then $\Q_{a,b,c}$ does not admit a maximal green sequence, but is mutation equivalent to a quiver which admits a maximal green sequence.
\end{coro}

The simplest such quiver is $\Q_{2,2,3}$.  For the benefit of the reader, Figure \ref{fig: toacyclic} shows a sequence of quiver mutations relating $\Q_{2,2,3}$ to the acyclic quiver $\Q_{0,2,1}$, and Figure \ref{fig: maxgreen} shows a maximal green sequence for quiver $\Q_{0,2,1}$.

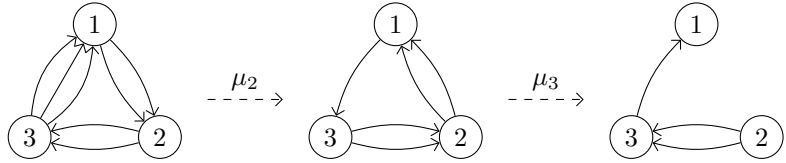
\begin{figure}[h!t]	
	\begin{tikzpicture}
	\begin{scope}
		\path[use as bounding box] (-1,-.5) rectangle  (1,1);
		\node[mutable] (1) at (90:1) {$1$};
		\node[mutable] (2) at (-30:1) {$2$};
		\node[mutable] (3) at (210:1) {$3$};
		\draw[-angle 90,relative, out=15,in=165] (1) to (2);
		\draw[-angle 90,relative, out=-15,in=-165] (1) to (2);
		\draw[-angle 90,relative, out=15,in=165] (2) to (3);
		\draw[-angle 90,relative, out=-15,in=-165] (2) to (3);
		\draw[-angle 90,relative, out=25,in=155] (3) to (1);
		\draw[-angle 90,relative] (3) to (1);
		\draw[-angle 90,relative, out=-25,in=-155] (3) to (1);
	\end{scope}
	\draw[-angle 90,dashed] (1.5,0) to node[above] {$\mu_2$} (2.5,0);
	\begin{scope}[xshift=4cm]
		\path[use as bounding box] (-1,-.5) rectangle  (1,1);
		\node[mutable] (1) at (90:1) {$1$};
		\node[mutable] (2) at (-30:1) {$2$};
		\node[mutable] (3) at (210:1) {$3$};
		\draw[angle 90-,relative, out=15,in=165] (1) to (2);
		\draw[angle 90-,relative, out=-15,in=-165] (1) to (2);
		\draw[angle 90-,relative, out=15,in=165] (2) to (3);
		\draw[angle 90-,relative, out=-15,in=-165] (2) to (3);
		\draw[angle 90-,relative, out=15,in=165] (3) to (1);
	\end{scope}
	\draw[-angle 90,dashed] (5.5,0) to node[above] {$\mu_3$} (6.5,0);
	\begin{scope}[xshift=8cm]
		\path[use as bounding box] (-1,-.5) rectangle  (1,1);
		\node[mutable] (1) at (90:1) {$1$};
		\node[mutable] (2) at (-30:1) {$2$};
		\node[mutable] (3) at (210:1) {$3$};
		\draw[-angle 90,relative, out=15,in=165] (2) to (3);
		\draw[-angle 90,relative, out=-15,in=-165] (2) to (3);
		\draw[-angle 90,relative, out=15,in=165] (3) to (1);
	\end{scope}
	\end{tikzpicture}
\caption{The quiver $\Q_{2,2,3}$ is mutation equivalent to $\Q_{0,2,1}$}
\label{fig: toacyclic}
\end{figure}

\section{Green-to-red sequences}

\subsection{Green-to-red sequences}

This section considers a weaker property than the existence of a maximal green sequence, which has some of the properties expected of maximal green sequences.

\begin{defn}
A \textbf{green-to-red sequence} for a quiver $\Q$ is a sequence of mutations which takes an initial $g$-seed with quiver $\Q$ to a $g$-seed with all vertices red.
\end{defn}
\noindent By Theorem \ref{thm: greenred}, the final $g$-seed in a green-to-red sequence has quiver $\Q$, and all $g$-vectors of the form $-e_k$.  The existence of a green-to-red sequence is equivalent to the existence of an all-red $g$-seed which is mutation-equivalent to an initial $g$-seed on $\Q$.

A maximal green sequence is clearly a green-to-red sequence.  However, a quiver may admit a green-to-red sequence but not admit a maximal green sequence.  In particular, $\Q_{2,2,3}$ admits a green-to-red sequence (Figure \ref{fig: greentored}).

\begin{figure}[h!t]	
	\begin{tikzpicture}
	\begin{scope}[xshift=2cm]
		\path[use as bounding box] (-1,-.5) rectangle  (1,1);
		\node[rounded rectangle,draw,fill=green!10] (1) at (90:1) {$1,0,0$};
		\node[rounded rectangle,draw,fill=green!10] (2) at (-25:1) {$0,1,0$};
		\node[rounded rectangle,draw,fill=green!10] (3) at (205:1) {$0,0,1$};
		\draw[-angle 90,relative, out=15,in=165] (1) to (2);
		\draw[-angle 90,relative, out=-15,in=-165] (1) to (2);
		\draw[-angle 90,relative, out=15,in=165] (2) to (3);
		\draw[-angle 90,relative, out=-15,in=-165] (2) to (3);
		\draw[-angle 90,relative, out=25,in=155] (3) to (1);
		\draw[-angle 90,relative] (3) to (1);
		\draw[-angle 90,relative, out=-25,in=-155] (3) to (1);
	\end{scope}
	\draw[black!50!green,-angle 90,dashed] (3.5,.25) to (4.5,.25);
	\begin{scope}[xshift=6cm]
		\path[use as bounding box] (-1,-.5) rectangle  (1,1);
		\node[rounded rectangle,draw,fill=green!10] (1) at (90:1) {$1,0,0$};
		\node[rounded rectangle,draw,fill=red!10] (2) at (-25:1) {$2,-1,0$};
		\node[rounded rectangle,draw,fill=green!10] (3) at (205:1) {$0,0,1$};
		\draw[angle 90-,relative, out=15,in=165] (1) to (2);
		\draw[angle 90-,relative, out=-15,in=-165] (1) to (2);
		\draw[angle 90-,relative, out=15,in=165] (2) to (3);
		\draw[angle 90-,relative, out=-15,in=-165] (2) to (3);
		\draw[angle 90-,relative, out=15,in=165] (3) to (1);
	\end{scope}
	\draw[black!50!green,-angle 90,dashed,out=0,in=45] (7.25,.25) to (9,-2.5);

	\begin{scope}[xshift=8cm,yshift=-3cm]
		\path[use as bounding box] (-1,-.5) rectangle  (1,1);
		\node[rounded rectangle,draw,fill=green!10] (1) at (90:1) {$1,0,0$};
		\node[rounded rectangle,draw,fill=red!10] (2) at (-25:1) {$2,-1,0$};
		\node[rounded rectangle,draw,fill=red!10] (3) at (205:1) {$1,0,-1$};
		\draw[-angle 90,relative, out=15,in=165] (2) to (3);
		\draw[-angle 90,relative, out=-15,in=-165] (2) to (3);
		\draw[-angle 90,relative, out=15,in=165] (3) to (1);
	\end{scope}
	\draw[black!50!red,-angle 90,dashed] (6.5,-2.75) to (5.5,-2.75);
	\begin{scope}[xshift=4cm,yshift=-3cm]
		\path[use as bounding box] (-1,-.5) rectangle  (1,1);
		\node[rounded rectangle,draw,fill=red!10] (1) at (90:1) {$0,0,-1$};
		\node[rounded rectangle,draw,fill=red!10] (2) at (-25:1) {$2,-1,0$};
		\node[rounded rectangle,draw,fill=green!10] (3) at (205:1) {$1,0,-1$};
		\draw[-angle 90,relative, out=15,in=165] (2) to (3);
		\draw[-angle 90,relative, out=-15,in=-165] (2) to (3);
		\draw[angle 90-,relative, out=15,in=165] (3) to (1);
	\end{scope}
	\draw[black!50!green,-angle 90,dashed] (2.5,-2.75) to (1.5,-2.75);
	\begin{scope}[xshift=0cm,yshift=-3cm]
		\path[use as bounding box] (-1,-.5) rectangle  (1,1);
		\node[rounded rectangle,draw,fill=red!10] (1) at (90:1) {$0,0,-1$};
		\node[rounded rectangle,draw,fill=green!10] (2) at (-25:1) {$0,1,-2$};
		\node[rounded rectangle,draw,fill=green!10] (3) at (205:1) {$1,0,-1$};
		\draw[angle 90-,relative, out=15,in=165] (2) to (3);
		\draw[angle 90-,relative, out=-15,in=-165] (2) to (3);
		\draw[angle 90-,relative, out=15,in=165] (3) to (1);
	\end{scope}
	\draw[black!50!green,-angle 90,dashed,out=-75,in=135] (0,-4.25) to (.75,-5.5);
	
	\begin{scope}[xshift=2cm,yshift=-6cm]
		\path[use as bounding box] (-1,-.5) rectangle  (1,1);
		\node[rounded rectangle,draw,fill=red!10] (1) at (90:1) {$0,0,-1$};
		\node[rounded rectangle,draw,fill=green!10] (2) at (-20:1) {$0,1,-2$};
		\node[rounded rectangle,draw,fill=red!10] (3) at (200:1) {$-1,0,0$};
		\draw[-angle 90,relative, out=15,in=165] (1) to (2);
		\draw[-angle 90,relative, out=-15,in=-165] (1) to (2);
		\draw[-angle 90,relative, out=15,in=165] (2) to (3);
		\draw[-angle 90,relative, out=-15,in=-165] (2) to (3);
		\draw[-angle 90,relative, out=15,in=165] (3) to (1);
	\end{scope}
	\draw[black!50!green,-angle 90,dashed] (3.5,-5.75) to (4.5,-5.75);
	\begin{scope}[xshift=6cm,yshift=-6cm]
		\path[use as bounding box] (-1,-.5) rectangle  (1,1);
		\node[rounded rectangle,draw,fill=red!10] (1) at (90:1) {$0,0,-1$};
		\node[rounded rectangle,draw,fill=red!10] (2) at (-20:1) {$0,-1,0$};
		\node[rounded rectangle,draw,fill=red!10] (3) at (200:1) {$-1,0,0$};
		\draw[angle 90-,relative, out=15,in=165] (1) to (2);
		\draw[angle 90-,relative, out=-15,in=-165] (1) to (2);
		\draw[angle 90-,relative, out=15,in=165] (2) to (3);
		\draw[angle 90-,relative, out=-15,in=-165] (2) to (3);
		\draw[angle 90-,relative, out=25,in=155] (3) to (1);
		\draw[angle 90-,relative] (3) to (1);
		\draw[angle 90-,relative, out=-25,in=-155] (3) to (1);
	\end{scope}
	
	\end{tikzpicture}
\caption{The quiver $\Q_{2,2,3}$ has green-to-red sequence}
\label{fig: greentored}
\end{figure}
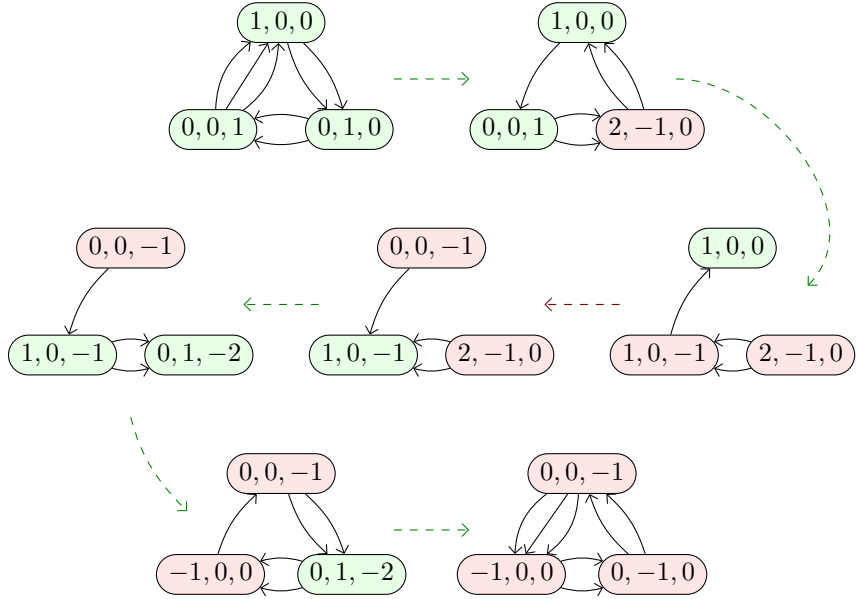

\begin{rem}
An equivalent condition to the existence of a green-to-red sequence appears in \cite[Prop.8.26]{GHKK}, together with important consequences for the scattering diagram (see Section \ref{section: gscattering} for the connection).
\end{rem}

%\subsection{Induced subquivers}

Like maximal green sequences, the existence of a green-to-red sequence is preserved under passing to an induced subquiver.

\begin{thm} \label{thm: inducedg2r}
If a quiver $\Q$ admits a green-to-red sequence, then any induced subquiver $\Q^\dagger$ admits a green-to-red sequence.
\end{thm}

The proof appears in Section \ref{section: proofs}.

\subsection{Conjugation}

By Theorem \ref{thm: greenred}, the quivers at the beginning and end of a green-to-red sequence are canonically identified.  This allows us to \textbf{conjugate} a green-to-red sequence for a quiver $\Q$ by any sequence of mutations from $\Q$ to another quiver $\Q'$.

\begin{thm}\label{thm: conjg2r}
The conjugation of a green-to-red sequence by any other sequence of mutations is a green-to-red sequence.
\end{thm}

\noindent Theorem \ref{thm: conjg2r} will be proved using scattering diagrams in Section \ref{section: proofs}.  As a corollary, the existence of green-to-red sequences is invariant under mutation.

%\begin{rem}
%The green-to-red sequence in Figure \ref{fig: greentored} is the conjugation of the maximal green sequence in Figure \ref{fig: maxgreen} by the sequence of mutations in Figure \ref{fig: toacyclic}.
%\end{rem}

\begin{coro}
If a green-to-red sequence exists for a quiver $\Q$, then a green-to-red sequence exists for any mutation equivalent quiver $\Q'$.
\end{coro}

%Finally, the following conjecture says that any pair of green-to-red sequences for a fixed quiver $\Q$ are related by `square moves' and `pentagon moves'.
%
%\begin{conj}
%For a quiver $\Q$, the union of the cones in the $g$-fan corresponding to seeds appearing in green-to-red sequences is simply connected.
%\end{conj}

%\subsection{The counterexample}
%
%%Consider the quiver $\Q_{2,2,3}$.  By Theorem \ref{thm: main}, it does not admit a 
%
%\begin{prop}
%The quiver $\Q_{2,2,3}$ does not admit a maximal green sequence, but is mutation equivalent to one that does.
%\end{prop}

%\section{Scattering diagrams}
%
%The main lemmas in the preceding article follow from the theory of \emph{scattering diagrams} for cluster algebras, recently introduced by Gross, Hacking, Keel and Kontsevich \cite{GHKK}.  This section reviews the necessary results.
%
%\subsection{Walls and scattering diagrams}

\section{Scattering diagrams}

\def\B{\mathsf{B}}
\def\D{\mathfrak{D}}

The remainder of this note reviews the scattering diagrams of Gross-Hacking-Keel-Kontsevich and the proofs of the necessary theorems.

\subsection{Formal elementary transformations}

Let $\Q$ be a quiver, and index the vertices of $\Q$ with the set $\{1,2,...,r\}$.  The quiver may be encoded in the \textbf{skew-adjacency matrix} $\B$, which is the following $r\times r$ matrix.
\[ \B_{ij} = (\text{\# of arrows from $j$ to $i$}) - (\text{\# of arrows from $i$ to $j$}) \]

Consider the ring
\[ R:= \mathbb{Q}[x_1^{\pm1},x_2^{\pm1},...,x_r^{\pm1}][[y_1,y_2,...,y_r]]\]
that is, formal power series in the variables $y_1,y_2,...,y_r$ with coefficients in the ring of Laurent polynomials in $x_1,x_2,...,x_r$. We use multinomial notation for the $x$ and $y$ variables; that is,
\[ \forall m=(m_1,m_2,...,m_r)\in \ZZ^r,\;\;\; x^m := \prod_{i=1}^r x_i^{m_i} ,\;\;\; \forall n=(n_1,n_2,...,n_r)\in \mathbb{N}^r,\;\;\; y^n := \prod_{i=1}^ry_i^{n_r}\]

For $n\in \mathbb{N}^r$, define the \textbf{formal elementary transformation} $E_{n}:R\longrightarrow R$ by
\[ E_{n}({x}^m) = (1+x^{\B n}y^n)^{\frac{n\cdot m}{\gcd(n)}}x^m,\;\;\; E_{n}(y^{n'})=y^{n'}  \]
Here, $n\cdot m=\sum_i n_im_i$ denotes the Euclidean inner product, and $\gcd(n)$ is the greatest common divisor of the coordinates of $n$.  This is an automorphism, with inverse
\[ E_{n}^{-1}({x}^m) = (1+x^{\B n}y^n)^{-\frac{n\cdot m}{\gcd(n)}}x^m,\;\;\; E_{n}^{-1}(y^{n'})=y^{n'}  \]
%Let $I$ be a monomial ideal in $\mathbb{Q}[[y_1,y_2,..,y_n]]$ with finite dimensional quotient.  Then $E_n$ descends to an 
%
%More generally, for $d\in \mathbb{Q}$, we may define a fractional power $E_{\B,n}^d:R \longrightarrow R$ by
%\[ E_{n}^d({x}^m) = (1+x^{\B n}y^n)^{d(n\cdot m)}x^m,\;\;\; E_{n}^d(y^{n'})=y^{n'}  \]
%It is clear that $E_{n}^dE_{n}^{d'}=E_{n}^{d+d'}$.  Since $E_{n}^0$ is trivial, each $E_n^d$ is an automorphism with inverse $E_n^{-d}$.	
Formal elementary transformations need not commute, with the following basic example.
\begin{prop}\label{prop: five}
Let $n,n'\in \mathbb{N}^r$ with $n\cdot \B n'=1$.  Then 
\[ E_nE_{n'} = E_{n'}E_{n+n'}E_{n} \]
\end{prop}
\begin{proof}  Since $\B$ is skew-symmetric, $n\cdot \B n=n'\cdot \B n'=0$ and $n'\cdot \B n= -n\cdot \B n'=-1$.  Furthermore, since $n\cdot\B n'=1$, both $\gcd(n)=\gcd(n')=1$.  In a similar argument, $(n+n')\cdot \B n'=n'\cdot \B n'+n\cdot \B n'=1$, and so $\gcd(n+n')=1$.  We now compute the two actions directly.
\begin{eqnarray*}
E_nE_{n'} (x^m) &=& E_n \left( (1+x^{\B n'}y^{n'})^{n'\cdot m} x^m\right) \\
&=& (1+(1+x^{\B n}y^{n})^{n\cdot B n'}x^{\B n'}y^{n'})^{n'\cdot m}(1+x^{\B n}y^{n})^{n\cdot m}x^m \\
&=& (1+x^{\B n'}y^{n'}+ x^{\B(n+n')}y^{n+n'})^{n'\cdot m}(1+x^{\B n}y^{n})^{n\cdot m}x^m
\end{eqnarray*}
For comparison, $E_{n'}E_{n+n'}E_n(x^m)$ is equal to
\begin{eqnarray*}
&=&E_{n'}E_{n+n'}\left( (1+x^{\B n}y^{n})^{n\cdot m} x^m\right) \\
&=& E_{n'}\left( (1+(1+x^{\B(n+n')}y^{n+n'})^{-1}x^{\B n}y^{n})^{n\cdot m} (1+x^{\B(n+n')}y^{n+n'})^{n\cdot m+n'\cdot m}x^m\right)\\
&=& E_{n'}\left( (1+x^{\B(n+n')}y^{n+n'}+x^{\B n}y^{n})^{n\cdot m} (1+x^{\B(n+n')}y^{n+n'})^{n'\cdot m}x^m\right)\\
&=& E_{n'}\left( (1+(1+x^{\B n'}y^{n'})x^{\B n}y^{n})^{n\cdot m} (1+x^{\B(n+n')}y^{n+n'})^{n'\cdot m}x^m\right)\\
&=& (1+x^{\B n}y^{n})^{n\cdot m} (1+(1+x^{\B n'}y^{n'})^{-1}x^{\B(n+n')}y^{n+n'})^{n'\cdot m}(1+x^{\B n'}y^{n'})^{n'\cdot m}x^m\\
&=& (1+x^{\B n}y^{n})^{n\cdot m} (1+x^{\B n'}y^{n'}+x^{\B(n+n')}y^{n+n'})^{n'\cdot m}x^m
\end{eqnarray*}
Since both actions send $y^n$ to $y^n$, they coincide on a generating set.
%\[ E_nE_{n'} (x^m) = E_n \left( (1+x^{\B n'}y^{n'})^{n'\cdot m} x^m\right)\]
%\[ = (1+(1+x^{\B n}y^{n})^{n\cdot B n'}x^{\B n'}y^{n'})^{n'\cdot m}(1+x^{\B n}y^{n})^{n\cdot m}x^m \]
%\[ = (1+x^{\B n'}y^{n'}+ x^{\B(n+n')}y^{n+n'})^{n'\cdot m}(1+x^{\B n}y^{n})^{n\cdot m}x^m \]
%
%\[ E_{n'}E_{n+n'}E_n(x^m) = E_{n'}E_{n+n'}\left( (1+x^{\B n}y^{n})^{n\cdot m} x^m\right)\]
%\[ = E_{n'}\left( (1+(1+x^{\B(n+n')}y^{n+n'})^{-1}x^{\B n}y^{n})^{n\cdot m} (1+x^{\B(n+n')}y^{n+n'})^{n\cdot m+n'\cdot m}x^m\right)\]
%\[ = E_{n'}\left( (1+x^{\B(n+n')}y^{n+n'}+x^{\B n}y^{n})^{n\cdot m} (1+x^{\B(n+n')}y^{n+n'})^{n'\cdot m}x^m\right)\]
%\[ = E_{n'}\left( (1+(1+x^{\B n'}y^{n'})x^{\B n}y^{n})^{n\cdot m} (1+x^{\B(n+n')}y^{n+n'})^{n'\cdot m}x^m\right)\]
%\[ = (1+x^{\B n}y^{n})^{n\cdot m} (1+(1+x^{\B n'}y^{n'})^{-1}x^{\B(n+n')}y^{n+n'})^{n'\cdot m}(1+x^{\B n'}y^{n'})^{n'\cdot m}x^m\]
%\[ = (1+x^{\B n}y^{n})^{n\cdot m} (1+x^{\B n'}y^{n'}+x^{\B(n+n')}y^{n+n'})^{n'\cdot m}x^m\]
\end{proof}

\begin{rem}
Following ideas of Reineke \cite{Rei10} and Kontsevich-Soibelman \cite{KS11}, any such identity between formal elementary transformations determines an analogous identity among quantum dilogarithms \cite{Kel11c}.  The above proposition corresponds to the Faddeev-Kashaev identity, which in turn implies the pentagonal identity for the Rogers' dilogarithm.
\end{rem}

%\[ E_n^2(x^m) = (1+x^{\B n}y^n)^{2n\cdot m} x^m \]
%\[ E_{2n}(x^m) = (1+x^{2\B n}y^{2n})^{2n\cdot m} x^m = (1+2x^{2\B n}y^{2n}+x^{4\B n}y^{4\B n})^{n\cdot m}x^m \]

\subsection{Scattering diagrams}

Scattering diagrams may be regarded as a method for visualizing diagrams of formal elementary transformations.  For a quiver $\Q$,\footnote{While $\Q$ does not appear in the definition of a wall, it is necessary to associate a formal elementary transformation $E_n$ to the wall, and so we shouldn't consider walls without first having a quiver in mind.} a \textbf{wall} is a pair $(n,W)$ consisting of
\begin{itemize}
	\item a non-zero element $n\in \mathbb{N}^r$, and
	\item a convex polyhedral cone $W$ in $\mathbb{R}^r$ which spans $n^\perp:=\{m\in \mathbb{R}^r \mid n\cdot m=0\}$.
%	\item a \textbf{multiplicity} $d\in \mathbb{Q}$.
\end{itemize}
Since $n\in \mathbb{N}^r$, this implies that $W\cap (\mathbb{R}_{>0})^r = W\cap (\mathbb{R}_{<0})^r=\emptyset$.  We will refer to the open half-space $\{m\in \mathbb{R}^r \mid n\cdot m>0\}$ as the \textbf{green} side of $W$, and $\{m\in \mathbb{R}^r \mid n\cdot m<0\}$ as the \textbf{red} side of $W$.  

A \textbf{scattering diagram} $\D$ is a collection of walls in $\mathbb{R}^r$ for the same $\Q$.\footnote{We explicitly allow multiple copies of the same wall.}  Let us say a smooth path $p:[0,1]\rightarrow \mathbb{R}^r$ in a scattering diagram $\D$ is \textbf{finite transverse} if
\begin{itemize}
	\item $p(0)$ and $p(1)$ are not in any walls,
	\item whenever the image of $p$ intersects a wall, it crosses it transversely, and
	\item the image of $p$ intersects finitely many walls, and does not intersect the boundary of a wall or the intersection of two walls which span different hyperplanes.
\end{itemize} 
A finite transverse path $p$ determines an automorphism of $R$, as follows.  List the walls crossed by $p$ in order:\footnote{The path $p$ may simultaneously cross multiple walls, but only if those walls span the same hyperplane.  In this case, the walls may be listed in any order, since the corresponding automorphisms commute.}
\[ (n_1,W_1), (n_2,W_2),...,(n_k,W_k) \]
and to each, associate the sign 
\[ \epsilon_i := \left\{
\begin{array}{cc}
 +1 &\text{ if $p$ crossed $W_i$ from the green side to the red side} \\
 -1 & \text{ if $p$ crossed $W_i$ from the red side to the green side}
\end{array}\right\} \]
Then the \textbf{path-ordered product} of $p$ is
\[ E_{n_k}^{\epsilon_k}E_{n_{k-1}}^{\epsilon_{k-1}}\cdots E_{n_1}^{\epsilon_1}\]
A finite scattering diagram is \textbf{consistent} if the path-ordered product associated to every finite transverse loop is the trivial automorphism of $R$.  Two finite scattering diagrams for the same quiver are \textbf{equivalent} if any path which is finite transitive in both diagrams determines the same path-ordered product.\footnote{Equivalence of scattering diagrams is a fairly restrictive condition; essentially, it just allows walls to be split into multiple walls, and vice versa.}

\begin{rem}
A finite scattering diagram determines a morphism diagram whose objects are a copy of $R$ for each chamber, and whose arrows are elementary transformations for each wall between chambers.  The consistency condition is equivalent to the commutativity of this diagram.
\end{rem}

%Ex. $\gmat{0 & -1 \\ 1 & 0}$.

\begin{ex}
Consider the quiver $\Q_1$ in Figure \ref{fig: quiver1}.

\begin{figure}[h!t]
	\begin{tabular}{c}
	\begin{subfigure}{.3\textwidth}
	\centering
	\begin{tikzpicture}
	\begin{scope}
		\path[use as bounding box] (-1,-.75) rectangle (1,1);
		\node[mutable] (1) at (-1,0) {$1$};
		\node[mutable] (2) at (1,0) {$2$};
		\draw[-angle 90] (1) to (2);
	\end{scope}
	\end{tikzpicture}	
    \caption{The quiver $\Q_1$}
    \label{fig: quiver1}
    	\end{subfigure}\\[1.75cm]
%	\vspace{1cm}
	\begin{subfigure}{.3\textwidth}
	\[ \B = \gmat{ 0 & -1 \\ 1 & 0 }\]
	\vspace{.25cm}
    \caption{The matrix $\B$}
    \label{fig: matrix1}
    	\end{subfigure}
	\end{tabular}
	\begin{subfigure}{.4\textwidth}
    \begin{tikzpicture}
    \begin{scope}[scale=.3]
    	\draw[step=1,draw=black!10,very thin] (-5.5,-5.5) grid (5.5,5.5);
        \draw[wall] (0,0) to (5.5,0);
        \draw[wall] (0,0) to (0,5.5);
        \draw[wall] (0,0) to (-5.5,0);
        \draw[wall] (0,0) to (5.5,-5.5);
        \draw[wall] (0,0) to (0,-5.5);
        \node[right] (1) at (5.5,0) {$((0,1),(\mathbb{R},0))$};
        \node[above] (2) at (0,5.5) {$((1,0),(0,\mathbb{R}))$};
%        \node[left] (3) at (-5.5,0) {$1+x^{(1,0)}$};
        \node[below] (4) at (5.5,-5.5) {$((1,1),\mathbb{R}_{\geq0}\cdot (1,-1))$};
%        \node[below] (5) at (0,-5.5) {$1+x^{(0,1)}$};
	\draw[green!50!black,->] (2,2) arc (45:390:2.828);
	\node[green!50!black, above left] at (-2,2) {$p$};
    \end{scope}
    \end{tikzpicture}
    \caption{The scattering diagram $\D$}
    \label{fig: scattering1}
    \end{subfigure}
    \caption{An example of a consistent finite scattering diagram}
    \label{fig: example1}
\end{figure}
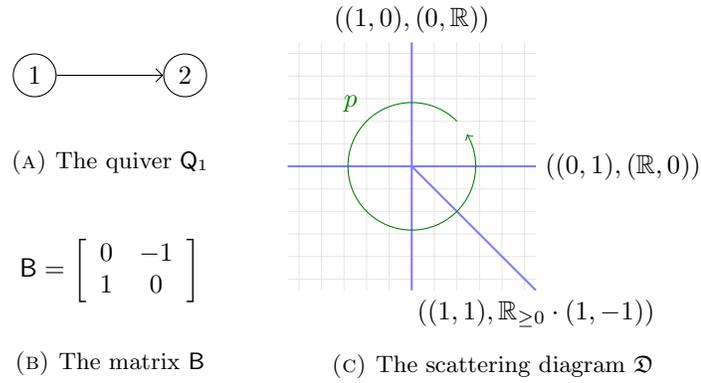
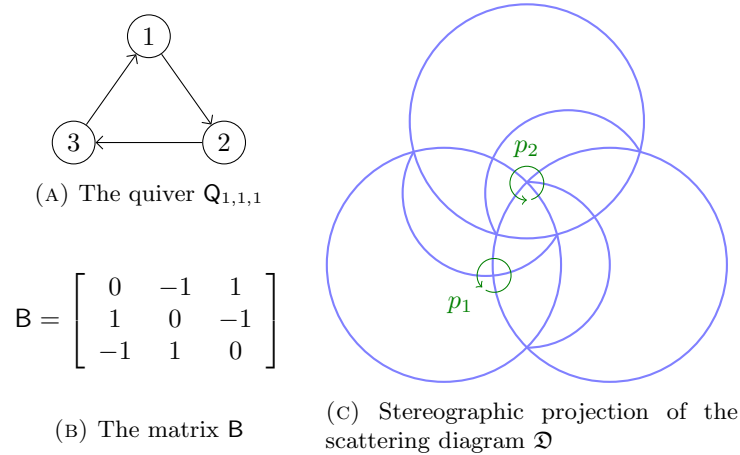

We claim the scattering diagram $\D$ in Figure \ref{fig: scattering1} is consistent.  The path-ordered product of the simple transverse path $p$ in the figure is 
\[ E_{(0,1)}^{-1}E_{(1,1)}^{-1}E_{(1,0)}^{-1}E_{(0,1)}E_{(1,0)} \]
By Proposition  \ref{prop: five}, this is the trivial automorphism of $R$.  All other path-ordered products of finite transverse loops are conjugations of this expression, and so they are also trivial.\footnote{Whenever $r=2$, checking the any finite transverse transverse loop around the origin suffices to check consistency.}  Hence, $\D$ is a consistent scattering diagram.
\end{ex}

\begin{ex}\label{ex: stereo1}
Consider the quiver $\Q_{1,1,1}$ in Figure \ref{fig: quiver2}.  Since $r=3$, a scattering diagram for $\Q_{1,1,1}$ consists of scale-invariant subsets of $\mathbb{R}^3$, it may be visualized by intersecting with a unit sphere and stereographically projecting onto the plane $x_1+x_2+x_3=\sqrt{3}$.

\begin{figure}[h!t]
	\begin{tabular}{c}
	\begin{subfigure}{.3\textwidth}
	\centering
	\begin{tikzpicture}
	\begin{scope}
%		\path[use as bounding box] (-1,-.75) rectangle (1,1);
		\node[mutable] (1) at (0,1.414) {$1$};
		\node[mutable] (2) at (1,0) {$2$};
		\node[mutable] (3) at (-1,0) {$3$};
		\draw[-angle 90] (1) to (2);
		\draw[-angle 90] (2) to (3);
		\draw[-angle 90] (3) to (1);
	\end{scope}
	\end{tikzpicture}	
    \caption{The quiver $\Q_{1,1,1}$}
    \label{fig: quiver2}
    	\end{subfigure}\\[1.75cm]
%	\vspace{1cm}
	\begin{subfigure}{.3\textwidth}
	\[ \B = \gmat{ 0 & -1 & 1 \\ 1 & 0 & -1 \\ -1 & 1 & 0 }\]
	\vspace{.25cm}
    \caption{The matrix $\B$}
    \label{fig: matrix2}
    	\end{subfigure}
	\end{tabular}
	\begin{subfigure}{.4\textwidth}
	\begin{tikzpicture}[scale=-.45,xscale=-1]
		\draw[wall] (0,-2.828) circle (3.46);
		\draw[wall] (2.45,1.414) circle (3.46);
		\draw[wall] (-2.45,1.414) circle (3.46);
		\begin{scope}
			\clip (-2.45,1.414) circle (3.46);
			\draw[wall] (-1.22,-.707) circle (2.45);
		\end{scope}
		\begin{scope}
			\clip (0,-2.828) circle (3.46);
			\draw[wall]  (1.22,-.707) circle (2.45);
		\end{scope}
		\begin{scope}
			\clip (2.45,1.414) circle (3.46);
			\draw[wall]  (0,1.414) circle (2.45);
		\end{scope}
		\draw[green!50!black,<-] (-1.32,2.08) arc (-225:120:.5);
		\node[below left,green!50!black] at (-1.32,2.08) {$p_1$};
		\draw[green!50!black,<-] (0,-.5) arc (-270:75:.5);
		\node[above,green!50!black] at (0,-1.5) {$p_2$};

    \end{tikzpicture}
    \caption{Stereographic projection of the scattering diagram $\D$}
    \label{fig: scattering2}
    \end{subfigure}
    \caption{An example of a consistent finite scattering diagram}
    \label{fig: example2}
\end{figure}

The scattering diagram $\D$ in Figure \ref{fig: scattering2} is consistent.  Checking this can be reduced to checking the path-ordered product of tiny loops around intersections of walls.  The intersections among the walls in $\D$ are the nine rays which project to the the nine intersections in Figure \ref{fig: scattering2}.  The triviality of the tiny loops around the $4$-valent intersections (such as $p_1$ in the figure) is the commutativity of the associated elementary transformations, while the triviality of the tiny loops around the $5$-valent intersections (such as $p_2$ in the figure) reduces to Proposition \ref{prop: five}.  %Hence, $\D$ is consistent.
%
%The path-ordered product of the simple transverse path $p$ in the figure is 
%\[ E_{(0,1)}^{-1}E_{(1,1)}^{-1}E_{(1,0)}^{-1}E_{(0,1)}E_{(1,0)} \]
%By Lemma \ref{lemma: relation}, this is the trivial automorphism of $R$.  All other path-ordered products of simple transverse loops are conjugations of this expression, and so they are also trivial.\footnote{Whenever $r=2$, checking a single simple transverse loop around the origin suffices to check consistency.}  Hence, $\D$ is a consistent scattering diagram.
\end{ex}

\subsection{Reduction of scattering diagrams}

The scattering diagrams we are interested in typically have infinitely many walls, and so there are not enough finite transverse loops to test for the `right' notion of consistency.  The simplest solution is to make sense of formal limits of consistent finite scattering diagrams.

Let $I$ be a monomial ideal in $\mathbb{Q}[[y_1,y_2,...,y_r]]$.  For each $n\in \mathbb{N}^r$, $E_n$ and $E_n^{-1}$ descend to well-defined automorphisms of $R/I$, which are trivial when $y^n\in I$.  A finite scattering diagram is \textbf{consistent mod $I$} if the path-ordered product associated to every transverse loop is the trivial automorphism of $R/I$.  Two finite scattering diagrams for the same quiver are \textbf{equivalent mod $I$} if any path which is finite transverse in both induces the same automorphism of $R/I$.

The \textbf{reduction} $\D/I$ of a scattering diagram $\D$ is obtained by deleting any wall of the form $(n,W)$ with $y^n\in I$.  If $\D$ is a consistent finite scattering diagram, then $\D/I$ is consistent mod $I$.  This idea can be extended to a criterion for the consistency of infinite scattering diagrams.

%\begin{defn}
A scattering diagram $\D$ is \textbf{consistent} if, for each monomial ideal $I\subset \mathbb{Q}[[y_1,y_2,...,y_r]]$ with finite dimensional quotient, the reduction $\D/I$ is finite and consistent mod $I$.  Two scattering diagrams $\D_1$ and $\D_2$ for the same quiver are \textbf{equivalent} if, for every monomial ideal $I\subset \mathbb{Q}[[y_1,y_2,...,y_r]]$ with finite dimensional quotient, the reductions $\D_1/I$ and $\D_2/I$ are finite and equivalent mod $I$.
%\end{defn}

We may now state one of the main existence and uniqueness results of the theory.

\begin{thm}\label{thm: DQ}\cite[Theorems 1.13 and 1.28]{GHKK}
For each quiver $\Q$, there is a consistent scattering diagram $\D(\Q)$, unique up to equivalence, such that 
\begin{itemize}
	\item for each $i\in \{1,2,...,r\}$, there is a wall of the form $(e_i, e_i^\perp)$,\footnote{That is, the hyperplane where the $i$th coordinate vanishes is a wall.} and
	\item every other wall $(n,W)$ in $\D(\Q)$ has the property that $\B n \not\in W$.
\end{itemize}
\end{thm}

\begin{rem}
\cite{GHKK} prove a more general theorem, which states that, given a scattering diagram $\D_{in}$ which is consistent modulo $I$, there is a unique (up to equivalence) consistent scattering diagram $\D$ such that $\D/I=\D_{in}$, and $\D\setminus \D_{in}$ consists of walls $(n,W)$ such that $\B n\not\in W$.  Theorem \ref{thm: DQ} is the case where $\D$ is the scattering diagram consisting of the coordinate hyperplanes, which is consistent modulo $\langle y_1,y_2,...,y_r\rangle ^2$.
\end{rem}

\begin{ex}
Consider the quiver $\Q_2$ in Figure \ref{fig: quiver3}.

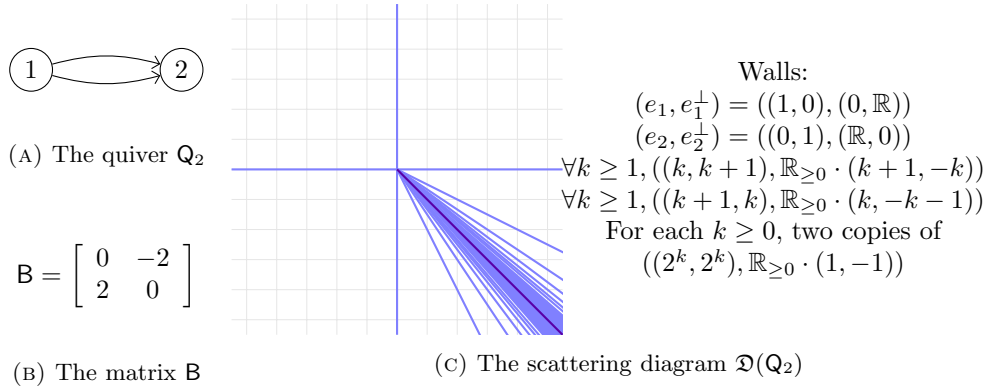
\begin{figure}[h!t]
	\begin{tabular}{c}
	\begin{subfigure}{.2\textwidth}
	\centering
	\begin{tikzpicture}
	\begin{scope}
		\path[use as bounding box] (-1,-.75) rectangle (1,1);
		\node[mutable] (1) at (-1,0) {$1$};
		\node[mutable] (2) at (1,0) {$2$};
		\draw[-angle 90,out=15,in=165] (1) to (2);
		\draw[-angle 90,out=-15,in=-165] (1) to (2);
	\end{scope}
	\end{tikzpicture}	
    \caption{The quiver $\Q_2$}
    \label{fig: quiver3}
    	\end{subfigure}\\[1.75cm]
%	\vspace{1cm}
	\begin{subfigure}{.2\textwidth}
	\[ \B = \gmat{ 0 & -2 \\ 2 & 0 }\]
	\vspace{.25cm}
    \caption{The matrix $\B$}
    \label{fig: matrix3}
    	\end{subfigure}
	\end{tabular}
\begin{subfigure}{.75\textwidth}
    \begin{tikzpicture}
    \begin{scope}[scale=.4]
    	\clip (-5.5,-5.5) rectangle (5.5,5.5);
    	\draw[step=1,draw=black!10,very thin] (-5.5,-5.5) grid (5.5,5.5);
        \draw[wall] (0,0) to (5.5,0);
        \draw[wall] (0,0) to (0,5.5);
        \draw[wall] (0,0) to (-5.5,0);
        \draw[wall] (0,0) to (0,-5.5);
        
        \draw[wall] (0,0) to (12,-6);
        \draw[wall] (0,0) to (12,-8);
        \draw[wall] (0,0) to (12,-9);
        \draw[wall] (0,0) to (12,-9.6);
        \draw[wall] (0,0) to (12,-10);
        \draw[wall] (0,0) to (12,-10.28);
        \draw[wall] (0,0) to (6,-12);
        \draw[wall] (0,0) to (8,-12);
        \draw[wall] (0,0) to (9,-12);
        \draw[wall] (0,0) to (9.6,-12);
        \draw[wall] (0,0) to (10,-12);
        \draw[wall] (0,0) to (10.28,-12);        
        \path[fill=blue!50] (0,0) to (10.5,-12) to (12,-10.5) to (0,0);
        \draw[wall, draw=gpurple] (0,0) to (5.5,-5.5);
    \end{scope}
    \begin{scope}[xshift=5cm]
    	\node at (0,0) {
	$\begin{array}{c}
	\text{Walls:} \\
	(e_1,e_1^\perp) = ((1,0),(0,\mathbb{R})) \\
	(e_2,e_2^\perp) = ((0,1),(\mathbb{R},0)) \\
	\forall k\geq 1, ((k,k+1),\mathbb{R}_{\geq0}\cdot (k+1,-k)) \\
	\forall k\geq 1, ((k+1,k),\mathbb{R}_{\geq0}\cdot (k,-k-1)) \\
	\text{For each $k\geq 0$, two copies of} \\
	%\forall k\geq 0, 2^{1-k}\cdot 
	((2^k,2^k),\mathbb{R}_{\geq0}\cdot (1,-1)) 
	\end{array}$
	};
    \end{scope}
    \end{tikzpicture}
    \caption{The scattering diagram $\D(\Q_2)$}
    \label{fig: scattering3}
    \end{subfigure}
    \caption{An example of a consistent infinite scattering diagram}
    \label{fig: example3}
\end{figure}

The scattering diagram in Figure \ref{fig: example3} is consistent. This fact is equivalent to the following infinite product identity, which is proven in \cite[Equation A.6]{GMN10}.
\[ E_{n_1}E_{n_2} = E_{n_2}E_{n_1+2n_2}E_{2n_1+3n_2}\cdots \left(\prod_{k=0}^\infty E_{2^k(n_1+n_2)}\right)^2
\cdots E_{3n_1+2n_2}E_{2n_1+n_2}E_{n_1}\]
%
%We claim without proof that the scattering diagram in Figure \ref{fig: example3} is consistent.  
The only walls $(n,W)$ such that $\B n\in W$ are the coordinate hyperplanes; hence, this scattering diagram is the unique (up to equivalence) scattering diagram $\D(\Q_2)$.

\end{ex}

\begin{ex}\label{ex: stereo2}
Consider the quiver $\Q_{1,1,2}$ in Figure \ref{fig: quiver4}.  Since $r=3$, a scattering diagram for $\Q_{1,1,2}$ consists of scale-invariant subsets of $\mathbb{R}^3$, it may be visualized by intersecting with a unit sphere and stereographically projecting onto the plane $x_1+x_2+x_3=\sqrt{3}$ (Figure \ref{fig: scattering4}).  In the figure, the purple arc is the projection of a half-plane which supports an infinite number of walls.

\begin{figure}[h!t]
	\begin{tabular}{c}
	\begin{subfigure}{.3\textwidth}
	\centering
	\begin{tikzpicture}
	\begin{scope}
%		\path[use as bounding box] (-1,-.75) rectangle (1,1);
		\node[mutable] (1) at (0,1.414) {$1$};
		\node[mutable] (2) at (1,0) {$2$};
		\node[mutable] (3) at (-1,0) {$3$};
		\draw[-angle 90] (1) to (2);
		\draw[-angle 90] (2) to (3);
		\draw[-angle 90,relative,out=15,in=165] (3) to (1);
		\draw[-angle 90,relative,out=-15,in=-165] (3) to (1);
	\end{scope}
	\end{tikzpicture}	
    \caption{The quiver $\Q_{1,1,2}$}
    \label{fig: quiver4}
    	\end{subfigure}\\[1.75cm]
%	\vspace{1cm}
	\begin{subfigure}{.3\textwidth}
	\[ \B = \gmat{ 0 & -1 & 2 \\ 1 & 0 & -1 \\ -2 & 1 & 0 }\]
	\vspace{.25cm}
    \caption{The matrix $\B$}
    \label{fig: matrix4}
    	\end{subfigure}
	\end{tabular}
	\begin{subfigure}{.4\textwidth}
	\begin{tikzpicture}[scale=-.45,xscale=-1,rotate=120]
		\draw[wall] (0,-2.828) circle (3.46);
		\draw[wall] (2.45,1.414) circle (3.46);
		\draw[wall] (-2.45,1.414) circle (3.46);
		\begin{scope}
			\clip (-2.45,1.414) circle (3.46);
			\draw[wall] (-1.22,-.707) circle (2.45);
		\end{scope}
		\begin{scope}
			\clip (0,-2.828) circle (3.46);
			\draw[wall]  (1.22,-.707) circle (2.45);
		\end{scope}

		\begin{scope}
			\clip (2.45,1.414) circle (3.46);
			\clip (0,1.414) circle (2.45);
			\draw[wall] (0,0) circle (2);
		\end{scope}
		\begin{scope}
			\clip (1.22,-.707) circle (2.45);
			\clip (0,-2.828) circle (3.46);
			\draw[wall] (0,0) circle (2);
		\end{scope}
		\begin{scope}
			\clip (0,0) circle (2);
			\begin{scope}
				\clip (-.816,1.414) circle (2.582);
				\draw[wall] (-.612,.354) circle (2.121);
			\end{scope}
			\begin{scope}
				\clip (-.490,1.414) circle (2.498);
				\draw[wall] (-.408,.707) circle (2.160);
			\end{scope}
			\begin{scope}
				\clip (-.350,1.414) circle (2.474);
				\draw[wall] (-.306,.884) circle (2.208);
			\end{scope}
		\end{scope}
		\begin{scope}
			\clip (0,-2.828) circle (3.46);
			\begin{scope}
				\clip (2.45,1.414) circle (3.46);
				\draw[wall] (.612,.354) circle (2.121);
			\end{scope}
			\begin{scope}
				\clip  (.816,1.414) circle (2.582);
				\draw[wall] (.408,.707) circle (2.160);
			\end{scope}
			\begin{scope}
				\clip (.490,1.414) circle (2.498);
				\draw[wall] (.306,.884) circle (2.208);
			\end{scope}
		\end{scope}
		\begin{scope}
			\clip (2.45,1.414) circle (3.46);
			\begin{scope}
				\clip (2.45,1.414) circle (3.46);
				\path[wall,fill=blue!50,even odd rule] (.223,1.414) circle (2.460) (-.223,1.414) circle (2.460);
			\end{scope}
			\draw[wall]  (-.816,1.414) circle (2.582);
			\draw[wall]  (-.490,1.414) circle (2.498);
			\draw[wall]  (-.350,1.414) circle (2.474);
			\draw[wall]  (-.272,1.414) circle (2.465);
			\draw[wall]  (-.223,1.414) circle (2.460);
			\draw[wall]  (.223,1.414) circle (2.460);
			\draw[wall]  (.272,1.414) circle (2.465);
			\draw[wall]  (.350,1.414) circle (2.474);
			\draw[wall]  (.490,1.414) circle (2.498);
			\draw[wall]  (.816,1.414) circle (2.582);
			\draw[wall, draw=gpurple]  (0,1.414) circle (2.45);
		\end{scope}
		%\draw[green!50!black,->] (-1.32,2.08) arc (-225:120:.5);
		%\node[above left,green!50!black] at (-1.32,2.08) {$p_1$};
		%\draw[green!50!black,->] (0,-.5) arc (-270:75:.5);
		%\node[below,green!50!black] at (0,-1.5) {$p_2$};

		%\node[below] at (-2.45,4.875) {$(e_1,e_1^\perp)$};
		%\node[below] at (2.45,4.875) {$(e_2,e_2^\perp)$};
		%\node[above] at (0,-6.289) {$(e_3,e_3^\perp)$};
		%\node[right] at (2.45,1.414) {$(e_1+e_2,\mathbb{R}_{\geq0}(e_2-e_1)+\mathbb{R}e_3)$};

    \end{tikzpicture}
    \caption{Stereographic projection of the scattering diagram $\D(\Q_{1,1,2})$}
    \label{fig: scattering4}
    \end{subfigure}
    \caption{An example of a consistent infinite scattering diagram}
    \label{fig: example4}
\end{figure}
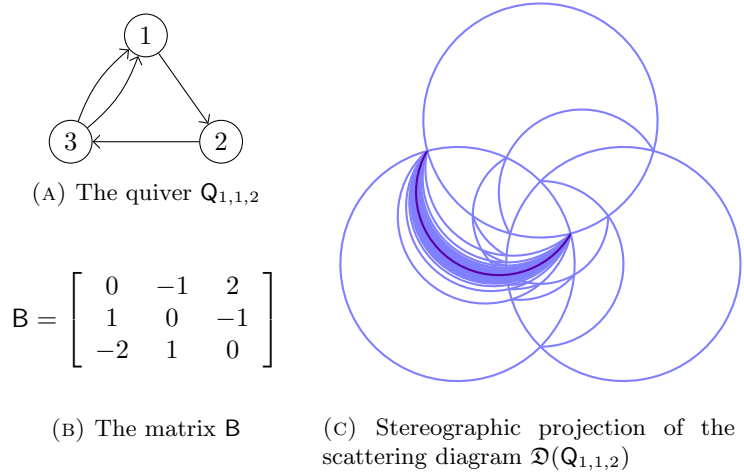

The scattering diagram $\D$ in Figure \ref{fig: scattering4} is consistent.  Again, checking this can be reduced to checking the path-ordered product of tiny loops around intersections of walls.  All but three of the intersections are $5$-valent, and correspond to Proposition \ref{prop: five}.  There is an intersection of an infinite number of walls at either end of the purple wall; the consistency around these intersections reduces to the consistency of Figure \ref{fig: scattering2}.  There is one final intersection between the purple ray and two blue walls; the formal elementary transformations of these walls all commute with each other, which implies consistency.
%
%The intersections among the walls in $\D$ are the nine rays which project to the the nine intersections in Figure \ref{fig: scattering2}.  The triviality of the tiny loops around the $4$-valent intersections (such as $p_1$ in the figure) reduces to Proposition \ref{prop: relations}, while the triviality of the tiny loops around the $5$-valent intersections (such as $p_2$ in the figure) reduces to Proposition \ref{prop: relations}.
\end{ex}

\subsection{Connection to $g$-vectors}\label{section: gscattering}

One remarkable feature of the scattering diagram $\D(\Q)$ is that it encodes every $g$-seed mutation equivalent to the initial seed on $\Q$.
%The $g$-seeds mutation-equivalent to the initial $g$-seed with quiver $\Q$ each span a simplicial cone in $\mathbb{R}^r$, each of which is the closure of a chamber of $\D(\Q)$.
A \textbf{chamber} in a scattering diagram $\D$ is a path-connected component of the complement $\mathbb{R}^r-\D$.  

\begin{ex}
Since no walls can pass through $(\mathbb{R}_{>0})^r$ and $(\mathbb{R}_{<0})^r$ and the coordinate hyperplanes are walls in $\D(\Q)$, it follows that the \textbf{all-positive chamber} $(\mathbb{R}_{>0})^r$ and the \textbf{all-negative chamber} $(\mathbb{R}_{<0})^r$ must be chambers in $\D(\Q)$.  
\end{ex}

Let us call a chamber in $\D(\Q)$ \textbf{reachable} if it can be connected to the all-positive chamber by a finite transverse path.

\begin{thm}\cite[Lemma 2.9]{GHKK}
Let $\{g_1,g_2,...,g_r\}$ be the $g$-vectors in a $g$-seed which is mutation-equivalent to the initial $g$-seed\footnote{Since the vertices of $\Q$ are already indexed by $\{1,2,...,r\}$, we use the initial seed corresponding to this indexing.} with quiver $\Q$.  Then 
\[ \mathbb{R}_{>0}g_1 + \mathbb{R}_{>0}g_2+...+\mathbb{R}_{>0}g_r\]
is a reachable chamber in $\D(\Q)$.  This induces a bijection between $g$-seeds mutation equivalent to the initial $g$-seed  on $\Q$, and reachable chambers of $\D(\Q)$.
\end{thm}

The theorem allows us to translate statements about $g$-seeds into statements about reachable chambers.  In what follows, by `$g$-seed', we mean a $g$-seed mutation-equivalent to the initial seed on $\Q$.

\begin{itemize}
	\item A vertex $k$ in a $g$-seed is green if and only if the $g$-vector $g_k$ is on the green side of the wall spanned by the other $g$-vectors.
	\item Two $g$-seeds are related by mutation if and only the corresponding reachable chambers share a facet.  The green mutation goes from the green side to the red side, and the red mutation goes from the red side to the green side.
	\item A $g$-seed with all green vertices must correspond to the all-positive chamber, and so it must be the initial $g$-seed.
	\item A $g$-seed with all red vertices must correspond to the all-negative chamber.  Hence, the existence of an all-red seed is equivalent to the all-negative chamber being reachable.
%	, and a green-to-red sequence is equivalent to a path from the all-positive chamber to the all-negative chamber which crosses finitely many walls.
%	\item A maximal green sequence is equivalent to a path from the all-positive chamber to the all-negative chamber which crosses finitely many walls, and always crosses from the green side to the red side.
\end{itemize}
%\begin{itemize}
%	\item A vertex in a $g$-seed is green if and only if the $g$-vector is on the green side of the opposite wall.
%	\item Two $g$-seeds are related by a single mutation if and only the corresponding reachable chambers share a facet.
%\end{itemize}

%\begin{coro}
%Let $\{g_k\}$ and $\{g_k'\}$ be two $g$-seeds mutation equivalent to the initial $g$-seed on $\Q$.
%\begin{itemize}
%	\item A vertex in a $g$-seed is green if and only if the $g$-vector is on the green side of the opposite wall.
%	\item Two $g$-seeds are related by mutation if and only if they share a wall.  The green mutation crosses from the green side to the red side, and the red mutation crosses from the red side to the green side.
%	\item The all-positive chamber corresponds to the initial $g$-seed.
%	\item An all-red $g$-seed must correspond to the all-negative chamber (reproving this result).  Hence, the existence of an all red seed is equivalent to the all-negative chamber being reachable.
%	\item A maximal green sequence is 
%\end{itemize}
%\end{coro}

%Mutation corresponds to crossing a wall; green means `starts on green side' and red means `starts on red side'.  Hence, a maximal green sequence is equivalent to a transverse path which only hits finitely many walls, and connects the all-positive orthant to the all-negative orthant.

A finite transverse path $p$ which begins in a reachable chamber defines a sequence of mutations of the corresponding $g$-seed, and every sequence of mutations of a $g$-seed can be encoded into a finite transverse path.\footnote{Specifically, there is a bijection between sequences of mutations of a $g$-seed, and finite transverse paths starting in the corresponding reachable chamber up to homotopies in the space of finite transitive paths.} Under this equivalence, the sequences of mutations we are interested in correspond to certain finite transverse paths in $\D(\Q)$.

\begin{itemize}
	\item A maximal green sequence is equivalent to a finite transitive path from the all-positive chamber to the all-negative chamber, which always crosses a wall from the green side to the red side.
	\item A green-to-red sequence is equivalent to a finite transitive path from the all-positive chamber to the all-negative chamber.
\end{itemize}
Hence, the existence of a green-to-red sequence on $\Q$ is equivalent to the all-negative chamber being reachable in $\D(\Q)$.

%\subsection{Proof of unique extension}
%
%Let $I\subset J$ be cofinite monomial ideals in $\mathbb{Q}[[y_1,y_2,...,y_r]]$, such that $J=I\oplus \mathbb{Q} y^n$.  Consider a finite scattering diagram $\D$ which is consistent mod $J$.  How can we add outgoing walls of the form $(n,W)$ to $\D$, to create a consistent scattering diagram mod $I$?
%
%For each $m\in n^\perp$, define a path in $\D$
%\[ p_m:[0,1] \rightarrow \mathbb{R}^r,\;\;\; p_m (t) = \sin(\pi t)n + \cos(\pi t)m \]
%This path is transverse, except when it passes through a joint.  When it is transverse, it defines an automorphism
%\[ \mu_m:R/I \rightarrow R/I\]
%
%Let $p$ be a tiny loop around a joint

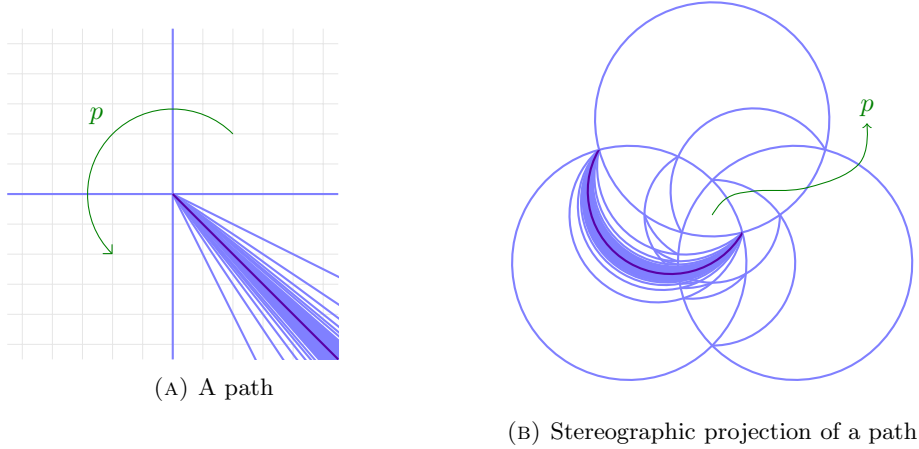
\begin{figure}[h!t]
	\begin{subfigure}{.4\textwidth}
    \begin{tikzpicture}
    \begin{scope}[scale=.4]
    	\clip (-5.5,-5.5) rectangle (5.5,5.5);
    	\draw[step=1,draw=black!10,very thin] (-5.5,-5.5) grid (5.5,5.5);
        \draw[wall] (0,0) to (5.5,0);
        \draw[wall] (0,0) to (0,5.5);
        \draw[wall] (0,0) to (-5.5,0);
        \draw[wall] (0,0) to (0,-5.5);
        \draw[wall] (0,0) to (12,-6);
        \draw[wall] (0,0) to (12,-8);
        \draw[wall] (0,0) to (12,-9);
        \draw[wall] (0,0) to (12,-9.6);
        \draw[wall] (0,0) to (12,-10);
        \draw[wall] (0,0) to (12,-10.28);
        \draw[wall] (0,0) to (6,-12);
        \draw[wall] (0,0) to (8,-12);
        \draw[wall] (0,0) to (9,-12);
        \draw[wall] (0,0) to (9.6,-12);
        \draw[wall] (0,0) to (10,-12);
        \draw[wall] (0,0) to (10.28,-12);        
        \path[fill=blue!50] (0,0) to (10.5,-12) to (12,-10.5) to (0,0);
        \draw[wall, draw=gpurple] (0,0) to (5.5,-5.5);
        
        \draw[green!50!black,-angle 90] (2,2) arc (45:225:2.828);
        \node[green!50!black,above left] at (-2,2) {$p$};
        %\draw[green!50!black,out=225,in=0] (3,3) to (0,2) to [out=180,in=90] (-2,0) to [out=-90,in=45] (-3,-3);
    \end{scope}
    \end{tikzpicture}
        \caption{A path}
    \label{fig: scattering5a}
	\end{subfigure}
    	\begin{subfigure}{.55\textwidth}
	\begin{center}
	\begin{tikzpicture}[scale=-.45,xscale=-1,rotate=120]
		\draw[wall] (0,-2.828) circle (3.46);
		\draw[wall] (2.45,1.414) circle (3.46);
		\draw[wall] (-2.45,1.414) circle (3.46);
		\begin{scope}
			\clip (-2.45,1.414) circle (3.46);
			\draw[wall] (-1.22,-.707) circle (2.45);
		\end{scope}
		\begin{scope}
			\clip (0,-2.828) circle (3.46);
			\draw[wall]  (1.22,-.707) circle (2.45);
		\end{scope}

		\begin{scope}
			\clip (2.45,1.414) circle (3.46);
			\clip (0,1.414) circle (2.45);
			\draw[wall] (0,0) circle (2);
		\end{scope}
		\begin{scope}
			\clip (1.22,-.707) circle (2.45);
			\clip (0,-2.828) circle (3.46);
			\draw[wall] (0,0) circle (2);
		\end{scope}
		\begin{scope}
			\clip (0,0) circle (2);
			\begin{scope}
				\clip (-.816,1.414) circle (2.582);
				\draw[wall] (-.612,.354) circle (2.121);
			\end{scope}
			\begin{scope}
				\clip (-.490,1.414) circle (2.498);
				\draw[wall] (-.408,.707) circle (2.160);
			\end{scope}
			\begin{scope}
				\clip (-.350,1.414) circle (2.474);
				\draw[wall] (-.306,.884) circle (2.208);
			\end{scope}
		\end{scope}
		\begin{scope}
			\clip (0,-2.828) circle (3.46);
			\begin{scope}
				\clip (2.45,1.414) circle (3.46);
				\draw[wall] (.612,.354) circle (2.121);
			\end{scope}
			\begin{scope}
				\clip  (.816,1.414) circle (2.582);
				\draw[wall] (.408,.707) circle (2.160);
			\end{scope}
			\begin{scope}
				\clip (.490,1.414) circle (2.498);
				\draw[wall] (.306,.884) circle (2.208);
			\end{scope}
		\end{scope}
		\begin{scope}
			\clip (2.45,1.414) circle (3.46);
			\begin{scope}
				\clip (2.45,1.414) circle (3.46);
				\path[wall,fill=blue!50,even odd rule] (.223,1.414) circle (2.460) (-.223,1.414) circle (2.460);
			\end{scope}
			\draw[wall]  (-.816,1.414) circle (2.582);
			\draw[wall]  (-.490,1.414) circle (2.498);
			\draw[wall]  (-.350,1.414) circle (2.474);
			\draw[wall]  (-.272,1.414) circle (2.465);
			\draw[wall]  (-.223,1.414) circle (2.460);
			\draw[wall]  (.223,1.414) circle (2.460);
			\draw[wall]  (.272,1.414) circle (2.465);
			\draw[wall]  (.350,1.414) circle (2.474);
			\draw[wall]  (.490,1.414) circle (2.498);
			\draw[wall]  (.816,1.414) circle (2.582);
			\draw[wall, draw=gpurple]  (0,1.414) circle (2.45);
		\end{scope}
		\draw[green!50!black,->] plot [smooth,tension=.7] coordinates {(0,0)  (-.653,-.126)  (-1.225,-.707) (-2.113,-2.034) (-3.46,-2.996) (-4.575,-2.641)};
		\node[green!50!black,above ] at (-4.575,-2.641) {$p$};
		%\draw[green] (0,0) to (-.551,-.318) to (-1.225,-.707)  to (-2.366,-2.049) to (-5.449,-3.146);
		%\draw[green!50!black,->] (-1.32,2.08) arc (-225:120:.5);
		%\node[above left,green!50!black] at (-1.32,2.08) {$p_1$};
		%\draw[green!50!black,->] (0,-.5) arc (-270:75:.5);
		%\node[below,green!50!black] at (0,-1.5) {$p_2$};

		%\node[below] at (-2.45,4.875) {$(e_1,e_1^\perp)$};
		%\node[below] at (2.45,4.875) {$(e_2,e_2^\perp)$};
		%\node[above] at (0,-6.289) {$(e_3,e_3^\perp)$};
		%\node[right] at (2.45,1.414) {$(e_1+e_2,\mathbb{R}_{\geq0}(e_2-e_1)+\mathbb{R}e_3)$};

    \end{tikzpicture}
    \end{center}
    \caption{Stereographic projection of a path}
    \label{fig: scattering5b}
    \end{subfigure}
    \caption{Paths corresponding to maximal green sequences}
    \label{fig: example5}
\end{figure}

\begin{ex}  For $\Q_2$ as in Figure \ref{fig: quiver3}, there is a unique maximal green sequence, which can be realized by the path $p$ in Figure \ref{fig: scattering5a}.  

The scattering diagram in Figure \ref{fig: scattering5a} also gives a visual proof of Lemma \ref{lemma: rank2}.  A finite transverse path cannot cross the purple ray, since it supports an infinite number of walls.\footnote{Additionally, it would have to cross an infinite number of walls to get to the purple ray.}  If a maximal green sequence could begin by mutating at vertex $2$, the corresponding path would start by crossing the ray $(\mathbb{R}_{\geq0},0)$.  The path cannot cross the purple wall and it cannot recross the ray $(\mathbb{R}_{\geq0},0)$, because that would correspond to a red mutation.  Hence, the path is trapped between the two, and can never reach the all-negative chamber.
\end{ex}

\begin{ex} For $\Q_{1,1,2}$ as in Figure \ref{fig: quiver4}, there are many maximal green sequences, one which can be realized by the path $p$ depicted in Figure \ref{fig: scattering5b}.  

The stereographic projection has been chosen so that the point $(1,1,1)$ maps to the origin, and so the green side of the stereographic projection of a wall is always the concave side.  A maximal green sequence for $\Q_{1,1,2}$ corresponds to a finite transverse path in $\D(\Q_{1,1,2})$ which travels from the inner-most chamber to the exterior, and only crosses walls from the concave side to the convex side.
\end{ex}

\section{Proofs}\label{section: proofs}

\subsection{Pullbacks of scattering diagrams}

Let $\Q$ be a quiver with vertex set $\{1,2...,r\}$, and $\Q^\dagger$ the induced subquiver on the subset $V \subset \{1,2,...,r\}$.  How can we describe $\D(\Q^\dagger)$ in terms of $\D(\Q)$?

The inclusion $V\subset\{1,2,...,r\}$ determines a coordinate projection $\pi:\mathbb{R}^r\rightarrow \mathbb{R}^V$ and coordinate inclusion $\pi^\top:\mathbb{R}^V\rightarrow \mathbb{R}^r$.  %Given a wall $(n,W)$ for the quiver $\Q^\dagger$, define its \textbf{pullback} to be the wall $(\pi^\top(n),\pi^{-1}(W))$ for $\Q$.
Define the \textbf{pullback} $\pi^{*}(\D(\Q^\dagger))$ of $\D(\Q^\dagger)$ to be the scattering diagram for $\Q$ a wall $(\pi^\top(n),\pi^{-1}(W))$ for each wall $(n,W)$ of $\D(\Q^{\dagger})$.

%\begin{lemma}
%The pullback of a consistent scattering diagram is consistent.
%\end{lemma}
%\begin{proof}
%
%\end{proof}

Our main result is that the pullback of $\D(\Q^\dagger)$ along the coordinate projection $\pi$ is the reduction of $\D(\Q)$ at the ideal generated by those $y_k$ such that $k\not\in V$.

\begin{thm}
Let $\Q$ be a quiver, let $\Q^\dagger$ be the induced subquiver with vertex set $V$, and let $\pi:\mathbb{R}^r \rightarrow \mathbb{R}^V$ be the projection onto the coordinates in $V$.  Then
\[ \pi^*\D(\Q^\dagger) = \D(\Q^\dagger) /\langle y_i \mid i\not\in V\rangle  \]
\end{thm}
%\begin{proof}
%For simplicity, we assume $V=\{1,2,...,v\}\subset \{1,2,...,r\}$, and denote elements of $m\in\mathbb{Z}^r$ by $m=(m_1,m_2)\in \mathbb{Z}^v\oplus \mathbb{Z}^{r-v}$.
%\end{proof}

\begin{proof}
For simplicity, we assume $V=\{1,2,...,v\}\subset \{1,2,...,r\}$, and denote elements of $m\in\mathbb{Z}^r$ by $m=(m_1,m_2)\in \mathbb{Z}^v\oplus \mathbb{Z}^{r-v}$.  The skew-adjacency matrix decomposes as
\[ \B = \gmat { \B^\dagger & -\mathsf{F}^\top \\ \mathsf{F}  & \mathsf{C}} \]
where $\B^\dagger$ is the skew-adjacency matrix of the induced subquiver $\Q^\dagger$.  

Let $R^\dagger:= \mathbb{Z}[x_1^{\pm1},...,x_v^{\pm1}][[y_1,...,y_v]]$ and define a ring homomorphism $\phi:R^\dagger\rightarrow R$ by
\[ \phi(x^m) = x^{(m,0)},\;\;\; \phi(y^n) = x^{(0,\mathsf{F} n)}y^{(n,0)} \]
We check that $\phi$ commutes with the pullback of walls, in that $E_n\circ \phi = \phi\circ E_{(n,0)}$. 
\[ E_{(n,0)}(\phi(x^m)) = (1+x^{(\B^\dagger n,\mathsf{F}n)}y^{(n,0)})^{n\cdot m} x^{(m,0)} = \phi(E_n(x^m))\] 
\[ E_{(n,0)}(\phi(y^{n'})) = E_{(n,0)}(x^{(0,\mathsf{F}n')}y^{(n',0)}) = \phi(E_n(y^{n'}))  \]

We now show that $\pi^*(\D(\Q^\dagger))$ is a consistent scattering diagram.  Let $I\subset \mathbb{Q}[[y_1,y_2,..,y_r]]$ be a monomial ideal with finite dimensional quotient, and let $J\subset R^\dagger$ be the monomial ideal generated by
\[ \{ y^n\in R^\dagger \mid \exists n'\in \mathbb{N}^{v-r} \text{ s.t. } y^{(n,n')}\in I\} \]
The induced $R$-ideal $R\otimes_{R^\dagger} J$ is the saturation $(I:\langle y_i\mid i\not\in V\rangle^\infty)$.
%
%(I:\langle y_i\mid i\not\in V\rangle) = \rangle y^{n} \mid \exists n'\in \mathbb{N}^{v-r} \text{ s.t. }y^{n+(0,n')}\in I \] 
%\mathbb{Q}[[y_1,y_2,..,y_r]]$ be the ideal generated by $I$ and $\{y_i \mid i\not\in V\rangle$.  
Since every wall in $\pi^*(\D(\Q^\dagger))$ is of the form $((n,0),\pi^{-1}(W))$ for some $n\in \mathbb{N}^v$, the two reductions $\pi^*(\D(\Q^\dagger))/I$ and $\pi^*(\D(\Q^\dagger))/R\otimes_{R^\dagger} J=\pi^*(\D(\Q^\dagger)/J)$ coincide.  %Hence, it suffices to check the consistency of $\pi^*(\D(\Q^\dagger))$.

Let $p$ be any transverse loop in $\pi^*(\D(\Q^\dagger))/I$, which is necessarily finite.  The walls crossed by $p$ are the preimages of the walls crossed by the loop $\pi\circ p$ in $\D(\Q^\dagger)/ J$.  Hence, if the path-ordered product of $\pi\circ p$ in $\D(\Q^\dagger)/J$ is
\[ E_{n_1}E_{n_2}\cdots E_{n_k} \]
then the path-ordered product of $p$ in $\D(\Q)/I$ is
\[ E_{(n_1,0)}E_{(n_2,0)}\cdots E_{(n_k,0)} \]
Any monomial $x^{(m_1,m_2)}\in R$ may be written as $\phi(x^{m_1})x^{(0,m_2)}$, and so
\[ E_{(n_1,0)}E_{(n_2,0)}\cdots E_{(n_k,0)}(\phi(x^{m_1})x^{(0,m_2)}) = x^{(0,m_2)}E_{(n_1,0)}E_{(n_2,0)}\cdots E_{(n_k,0)}(\phi(x^{m_1}))\]
\[= x^{(0,m_2)}\phi(E_{n_1}E_{n_2}\cdots E_{n_k}(x^{m_1}))= x^{(0,m_2)}\phi(x^{m_1}) = x^{(m_1,m_2)}\]
Hence, the path-ordered product of $p$ is trivial on $R/I$, and so $\pi^*(\D(\Q^\dagger))$ is consistent.

Both $\pi^*(\D(\Q^\dagger))$ and $\D(\Q)/\langle y_i \mid i\not\in V\rangle$ are consistent scattering diagrams in $\mathbb{R}^r$ for the matrix $\mathsf{B}$ with the same set of incoming walls $\{(e_i,e_i^\perp) \mid i\in V\}$.  By the uniqueness theorem \cite[Theorem 1.7]{GHKK}, they must coincide.
%Since every wall in $\pi^*(\D(\Q^\dagger))$ is of the form $E_{(n,0)$, the path ordered product of $p$ 
%
%%Walls in $\Dof the form $((n,0),W)$, and so
%
%%A finite transverse loop $p$ in $\pi^*(\D(\Q))$ projects to a 
%%\[ E_{n_k}^{\epsilon_k}E_{n_{k-1}}^{\epsilon_{k-1}}\cdots E_{n_1}^{\epsilon_1}\]
%
%%\[ E_{\pi^\top(n)}(x^m) = (1+x^{\mathsf{B} \pi^\top n}y^{\pi^\top n})^{\pi^\top(n)\cdot m} x^m \]
%%\[ = (1+x^{\mathsf{B} \pi^\top n}\pi^*(y^{n}))^{n\cdot \pi(m)} x^m \]
%
%
%
%The action of a pullback is
%\[ E_{(n,0)}(x^{(m_1,m_2)}) = (1+x^{(\B^\dagger n,\mathsf{F} n)}y^{(n,0)})^{n\cdot m_1} x^{(m_1,m_2)} \]
%%\[ = (1+x^{\pi^\top\mathsf{B}^\dagger  n+\mathsf{F} n}\pi^*(y^{\pi^\top n}))^{n\cdot \pi(m)} x^m \]
%Define two maps $f:R\rightarrow R^\dagger$ and $g:R^\dagger\rightarrow R$, where
%\[ g(x^m) = x^{(m,0)},\;\;\; g(y^n) = x^{(0,\mathsf{F} n)}y^{(n,0)} \]
%Then the above computation becomes
%\[ g(E_n(x^m)) = E_{(n,0)}(g(x^m)) \]
%\[ E_{(n,0)}(x^{(m_1,m_2)}) = g(E_n(x^{m_1})) x^{m_2} \]
%Consider a finite transverse loop $p$ in $\pi^*(\D(\Q^\dagger))/J$.  The projection $\pi\circ p$ is a finite transverse loop in $\D(\Q^\dagger)/J$, and so the path-order product on $R^\dagger/J$ is trivial.  By the above computation, the path-ordered product of $p$ is on $R/J$ is trivial.  Hence, $\pi^*(\D(\Q^\dagger))$ is consistent.
%
%To show that $\pi^*(\D(\Q^\dagger))=\D(\Q^\dagger)/\langle y_i \mid i\not\in V\rangle $, we simply need to show that their set of incoming walls coincide.
\end{proof}

\begin{rem}
The scattering diagram $\pi^*(\D(\Q^\dagger))$ can be regarded as the scattering diagram of a \emph{quiver with frozen vertices}; that is, vertices where mutation is prohibited.  Here,  the quiver is $\Q$ with frozen vertices $V^c$.  The theorem may be regarded as saying that freezing a vertex $k$ reduces the corresponding scattering diagram by the ideal $\langle y_k\rangle $, and adding a frozen vertex pulls back the corresponding scattering diagram.
\end{rem}

%\begin{coro}
%\end{coro}

%\begin{coro}
%If $m\in \mathbb{R}^r$ is in a reachable chamber of $\D(\Q)$, then $\pi(m)$ is in a reachable chamber of $\D(\Q^\dagger)$.  If $p$ is a finite transverse path in $\D(\Q)$, then $\pi\circ p$ is a finite transverse path in $\D(\Q)$.
%\end{coro}

\begin{proof}[Proof of Theorem \ref{thm: induced}]
Given a maximal green sequence on $\Q$, there is a finite transverse path $p$ in $\D(\Q)$ which passes through the corresponding reachable chambers.  In particular, $p$ begins in the all-positive chamber and ends in the all-negative chamber.  The projection $\pi\circ p$ is a finite transverse path in $\D(\Q^\dagger)$ which starts in the all-positive chamber and ends in the all-negative chamber, and always crosses walls from the green side to the red side.  Hence, the sequence of walls crossed by $\pi \circ p$ determines a maximal green sequence on  the induced subquiver $\Q^\dagger$.

If a maximal green sequence on $\Q$ begins with a sequence of $k$-many mutations at vertices in $V$, then the path $p$ begins by crossing $k$-many walls in $\pi^*(\D(\Q^\dagger))\subset \D(\Q)$.  Then the projected path $\pi\circ p$ begins by crossing the images of those $k$-many walls in $\D(\Q^\dagger)$.  The maximal green sequence associated to $\pi\circ p$ then begins with the same $k$-many mutations as the original maximal green sequence.
\end{proof}

\begin{proof}[Proof of Theorem \ref{thm: inducedg2r}]
The proof of the analogous result for green-to-red sequences is virtually identical.  The only difference is that the corresponding finite transverse paths need not always cross from the green side to the red side.
\end{proof}

\begin{rem}
Once a path $p$ in $\D(\Q)$ from a maximal green sequence crosses a wall not in $\pi^*(\D(\Q^\dagger))$, it becomes quite difficult to translate mutations in the original maximal green sequence on $\Q$ into mutations in the induced maximal green sequence on $\Q^\dagger$.
\end{rem}

\subsection{Scattering diagrams and mutation}

Given a quiver $\Q$ and its mutation $\mu_k(\Q)$ at a vertex $k$, there is a natural relation between the corresponding scattering diagrams.

Let $\mathsf{E}_{k}$ be the elementary $r\times r$-matrix with a $1$ in the $(k,k)$ entry and $0$s elsewhere, and let $\B^+$ denote the matrix $r\times r$ matrix with the same positive entries as $\B$ and zero otherwise.  Let $\B^-:=\B^+-\B= (\B^+)^\top$.  Define a pair of linear maps on $\mathbb{R}^r$ as follows.
\[\mathsf{G}_k := (\mathsf{Id}-2\mathsf{E}_k)(\mathsf{Id}+\B^+\mathsf{E}_k),\;\;\; \mathsf{R}_k := (\mathsf{Id}-2\mathsf{E}_k)(\mathsf{Id}+\B^-\mathsf{E}_k) \]
These are both involutions which fix the hyperplane $e_k^\perp$.  Their transposes are given by
\[\mathsf{G}_k^\top = (\mathsf{Id}-2\mathsf{E}_k)(\mathsf{Id}-\mathsf{E}_k\B^-),\;\;\; \mathsf{R}_k^\top = (\mathsf{Id}-2\mathsf{E}_k)(\mathsf{Id}-\mathsf{E}_k\B^+) \]

The following lemma describes how the scattering diagram $\D(\Q)$ changes when mutating a vertex in $\Q$.

\begin{lemma}\cite[Theorem 1.33]{GHKK}
Given a quiver $\Q$ and a vertex $k$, the scattering diagram $\D(\mu_k(\Q))$ is equivalent to the scattering diagram $\D'$ constructed as follows.
\begin{itemize}
	\item For each $(n,W)\in \D$, if $W_+:=\{w\in W \mid e_k\cdot w \geq 0\}$ spans a hyperplane, then
	\[ (\mathsf{G}^\top_k(n),\mathsf{G}_k(W_+)) \]
	is a wall in $\D'$.
	\item For each $(n,W)\in \D$, if $W_-:=\{w\in W \mid e_k\cdot w \leq 0\}$ spans a hyperplane, then
	\[ (\mathsf{R}^\top_k(n),\mathsf{R}_k(W_-)) \]
	is a wall in $\D'$.
	\item For each $(n,W)\in \D$, if $W$ is contained in $e_k^\perp$, then $(n,W)$ also a wall in $\D'$.
\end{itemize}
%
%
%Given a quiver $\Q$ vertex $k$, there is an equivalence of scattering diagrams
%\[ \mu_k(\D(\Q)) \simeq \D(\mu_k(\Q)) \]
\end{lemma}
\noindent As a corollary, the piece-wise linear map
\[ T_k:m\mapsto \left\{\begin{array}{cc}
\mathsf{G}_k(m) & \text{if } e_k\cdot m \geq0\\
\mathsf{R}_k(m) &\text{if } e_k\cdot m \leq0
\end{array}\right\} = \left\{\begin{array}{cc}
(\mathsf{Id}-2\mathsf{E}_k)(\mathsf{Id}+\B^+\mathsf{E}_k)m & \text{if } e_k\cdot m \geq0\\
(\mathsf{Id}-2\mathsf{E}_k)(\mathsf{Id}+\B^-\mathsf{E}_k)m&\text{if } e_k\cdot m \leq0
\end{array}\right\} \]
sends the support of walls in $\D(\Q)$ to the support of walls in $\Q(\mu_k(\Q))$, and thus chambers of $\D(\Q)$ to chambers of $\D(\mu_k(\Q))$.

\begin{rem}
The construction of $\D'\simeq \D(\mu_k(\Q))$ appears different in \cite{GHKK}, because they consider scattering diagrams with normal vectors $n$ in a simplicial cone $P\subset \mathbb{Z}^r$, rather than $\mathbb{N}^r$. Composing the construction in \cite{GHKK} with a linear map which takes their $P$ to $\mathbb{N}^r$ recovers the above construction of $\D'$.
\end{rem}

%\begin{rem}
%As remarked in \cite[???]{GHKK}, this is a tropicalization of the mutation map.
%\end{rem}

\begin{proof}[Proof of Theorem \ref{thm: conjg2r}]
First, we consider the conjugation of a green-to-red sequence by a single mutation of $\Q$ at vertex $k$.

A green-to-red sequence on $\Q$ corresponds to a finite transverse path $p$ in $\D(\Q)$ which begins in the all-positive chamber and ends in the all-negative chamber.  The image $T_k\circ p$ is a finite transitive path in $\D(\mu_k(\Q))$.   The path $T_k\circ p$ begins in the chamber which shares the wall $(e_k,e_k^\perp)$ with the all-positive chamber and ends in the chamber which shares the wall $(e_k,e_k^\perp)$ with the all-negative chamber.  Choose a path $p_1$ from the all-positive chamber to the beginning of $T_k\circ p$ which only crosses $(e_k,e_k^\perp)$ once and does not touch any other wall.  Similarly, choose a path $p_2$ from the end of $T_k\circ p$ to the all-negative chamber which only crosses $(e_k,e_k^\perp)$ once and does not touch any other wall.  

There is a finite transitive path which travels along $p_1$, then travels along $T_k\circ p$, and then travels along $p_2$.\footnote{Some smoothing may be required at the meeting points between paths.}  This path begins in the all-positive chamber and  ends in the all-negative chamber, so it determines a green-to-red sequence on $\Q^\dagger$.  Specifically, the sequence of mutations it determines is the conjugation of the original green-to-red sequence.  

By iterating this argument for an arbitrary sequence of mutations on $\Q$, the general theorem is proven.
\end{proof}

\subsection{Visualizing $\D(\Q_{2,2,3})$}

We can now deduce enough about the scattering diagram $\D(\Q_{2,2,3})$ to visualize why there are no maximal green sequences for $\Q_{2,2,3}$, without having to compute the entire scattering diagram.  As in Examples \ref{ex: stereo1} and \ref{ex: stereo2}, scattering diagrams in $\mathbb{R}^3$ may be visualized by intersecting with a unit sphere and stereographically projecting onto the plane $x_1+x_2+x_3=\sqrt{3}$ .

%\begin{ex}
%Consider the quiver $\Q_{2,2,3}$.  
%Computing the scattering diagram $\D(\Q_{2,2,3})$ is fairly very complicated, but we do not need the whole scattering diagram to rule out the existence of maximal green sequences.  
For any vertex $i$, the reduced scattering diagram $\D(\Q_{2,2,3})/\langle y_i\rangle$ is the pullback of the scattering diagram of a quiver with 2 vertices.  These three reductions of $\D(\Q_{2,2,3})$ are depicted in Figure \ref{fig: reductions}.  Each reduction contains a half-plane which supports an infinite number of walls; these are purple in the figure (the second reduction has an infinite number of half-planes which support an infinite number of walls; these are dense in the purple region).  

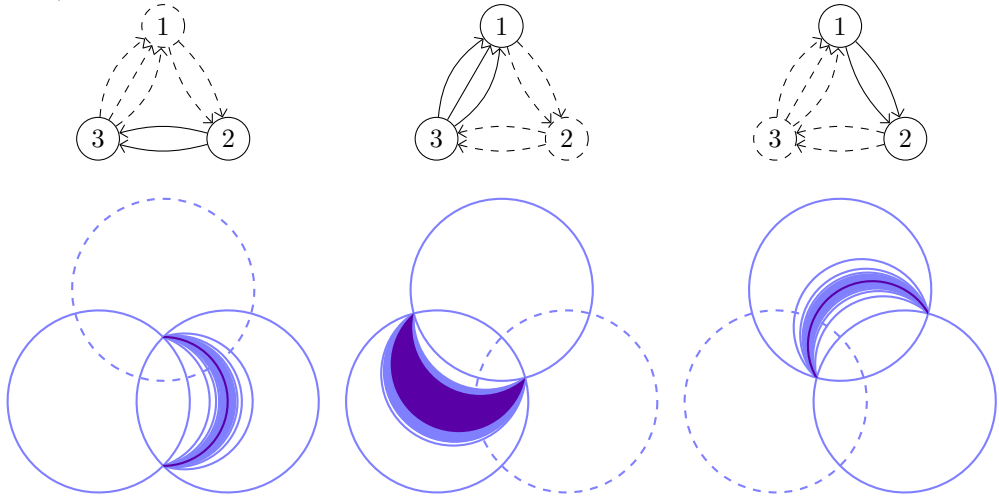
\begin{figure}[h!t]
\begin{tikzpicture}
	\begin{scope}[scale=-.35,xshift=0cm,xscale=-1,rotate=120]
		\draw[wall] (0,-2.828) circle (3.46);
		\draw[wall] (2.45,1.414) circle (3.46);
		\draw[wall,dashed] (-2.45,1.414) circle (3.46);
		\begin{scope}
			\clip (0,-2.828) circle (3.46);
			\draw[wall]  (1.633,0) circle (2.582);
			\draw[wall]  (1.47,-.283) circle (2.498);
			\draw[wall]  (1.4,-.404) circle (2.474);
			\draw[wall]  (1.361,-.471) circle (2.465);
			\draw[wall]  (1.336,-.514) circle (2.460);
			\path[wall,fill=blue!50,even odd rule] (1.336,-.514) circle (2.460) (1.113,-.9) circle (2.460);
			\draw[wall]  (1.113,-.9) circle (2.460);
			\draw[wall]  (1.089,-.943) circle (2.465);
			\draw[wall]  (1.050,-1.01) circle (2.474);
			\draw[wall]  (.98,-1.131) circle (2.498);
			\draw[wall]  (.816,-1.414) circle (2.582);
			\draw[wall, draw=gpurple]  (1.225,-.707) circle (2.45);
		\end{scope}
	\end{scope}
	\begin{scope}[xshift=4.5cm,scale=-.35,xscale=-1]%rotate=240]
		\draw[wall] (0,-2.828) circle (3.46);
		\draw[wall,dashed] (2.45,1.414) circle (3.46);
		\draw[wall] (-2.45,1.414) circle (3.46);
		\begin{scope}
			\clip (-2.45,1.414) circle (3.46);
			\draw[wall]  (-1.837,0.354) circle (2.739);
			\draw[wall]  (-1.781,.257) circle (2.691);
			\draw[wall]  (-1.774,.244) circle (2.684);
			\path[fill=blue!50,even odd rule] (-1.774,.244) circle (2.684) (-.676,-1.658) circle (2.684);
			\draw[wall]  (-.676,-1.658) circle (2.684);
			\draw[wall]  (-.668,-1.671) circle (2.691);
			\draw[wall]  (-.612,-1.768) circle (2.739);
			\path[fill=gpurple,even odd rule] (-1.633,0) circle (2.582) (-.816,-1.414) circle (2.582);
		\end{scope}
	\end{scope}
	\begin{scope}[xshift=9cm,scale=-.35,xscale=-1,rotate=240]
		\draw[wall,dashed] (0,-2.828) circle (3.46);
		\draw[wall] (2.45,1.414) circle (3.46);
		\draw[wall] (-2.45,1.414) circle (3.46);
		\begin{scope}
			\clip (2.45,1.414) circle (3.46);
			\draw[wall]  (-.816,1.414) circle (2.582);
			\draw[wall]  (-.490,1.414) circle (2.498);
			\draw[wall]  (-.350,1.414) circle (2.474);
			\draw[wall]  (-.272,1.414) circle (2.465);
			\draw[wall]  (-.223,1.414) circle (2.460);
			\path[wall,fill=blue!50,even odd rule] (.223,1.414) circle (2.460) (-.223,1.414) circle (2.460);
			\draw[wall]  (.223,1.414) circle (2.460);
			\draw[wall]  (.272,1.414) circle (2.465);
			\draw[wall]  (.350,1.414) circle (2.474);
			\draw[wall]  (.490,1.414) circle (2.498);
			\draw[wall]  (.816,1.414) circle (2.582);
			\draw[wall, draw=gpurple]  (0,1.414) circle (2.45);
		\end{scope}
	\end{scope}
	\begin{scope}[yshift=3.5cm]
		\path[use as bounding box] (-1,-.5) rectangle  (1,1);
		\node[mutable,dashed] (1) at (90:1) {$1$};
		\node[mutable] (2) at (-30:1) {$2$};
		\node[mutable] (3) at (210:1) {$3$};
		\draw[dashed,-angle 90,relative, out=15,in=165] (1) to (2);
		\draw[dashed,-angle 90,relative, out=-15,in=-165] (1) to (2);
		\draw[-angle 90,relative, out=15,in=165] (2) to (3);
		\draw[-angle 90,relative, out=-15,in=-165] (2) to (3);
		\draw[dashed,-angle 90,relative, out=25,in=155] (3) to (1);
		\draw[dashed,-angle 90,relative] (3) to (1);
		\draw[dashed,-angle 90,relative, out=-25,in=-155] (3) to (1);
	\end{scope}
	\begin{scope}[yshift=3.5cm,xshift=4.5cm]
		\path[use as bounding box] (-1,-.5) rectangle  (1,1);
		\node[mutable] (1) at (90:1) {$1$};
		\node[mutable,dashed] (2) at (-30:1) {$2$};
		\node[mutable] (3) at (210:1) {$3$};
		\draw[dashed,-angle 90,relative, out=15,in=165] (1) to (2);
		\draw[dashed,-angle 90,relative, out=-15,in=-165] (1) to (2);
		\draw[dashed,-angle 90,relative, out=15,in=165] (2) to (3);
		\draw[dashed,-angle 90,relative, out=-15,in=-165] (2) to (3);
		\draw[-angle 90,relative, out=25,in=155] (3) to (1);
		\draw[-angle 90,relative] (3) to (1);
		\draw[-angle 90,relative, out=-25,in=-155] (3) to (1);
	\end{scope}
	\begin{scope}[yshift=3.5cm,xshift=9cm]
		\path[use as bounding box] (-1,-.5) rectangle  (1,1);
		\node[mutable] (1) at (90:1) {$1$};
		\node[mutable] (2) at (-30:1) {$2$};
		\node[mutable,dashed] (3) at (210:1) {$3$};
		\draw[-angle 90,relative, out=15,in=165] (1) to (2);
		\draw[-angle 90,relative, out=-15,in=-165] (1) to (2);
		\draw[dashed,-angle 90,relative, out=15,in=165] (2) to (3);
		\draw[dashed,-angle 90,relative, out=-15,in=-165] (2) to (3);
		\draw[dashed,-angle 90,relative, out=25,in=155] (3) to (1);
		\draw[dashed,-angle 90,relative] (3) to (1);
		\draw[dashed,-angle 90,relative, out=-25,in=-155] (3) to (1);
	\end{scope}
\end{tikzpicture}
    \caption{Stereographic projection of three reductions of $\D(\Q_{2,2,3})$}
    \label{fig: reductions}
\end{figure}

The stereographic projection has been chosen so that the point $(1,1,1)$ maps to the origin, and so the green side of the stereographic projection of a wall is always the concave side.  A maximal green sequence for $\Q_{2,2,3}$ corresponds to a finite transverse path in $\D(\Q_{2,2,3})$ which travels from the inner-most chamber to the exterior, and only crosses walls from the concave side to the convex side.

Consider a finite transverse path which starts in the inner-most chamber, and only crosses walls from the concave side to the convex side.  As soon as it crosses a wall, the path is now trapped between the convex side of the wall it just crossed (which is a coordinate hyperplane, and hence a circle in the projection) and a purple half-space.  Hence, it can never reach the exterior chamber, and so a maximal green sequence cannot exist.

%By a similar argument to Exercise ???, any sequence of green mutations corresponds to a path which gets stuck between a purple wall and the convex side of a wall, and can never reach the all-green chamber.  Hence, a maximal green sequence cannot exist.

\begin{figure}[h!t]
	\begin{tikzpicture}[scale=-.45,xscale=-1]%,rotate=-60]
	\begin{scope}
		\draw[wall] (0,-2.828) circle (3.46);
		\draw[wall] (2.45,1.414) circle (3.46);
		\draw[wall] (-2.45,1.414) circle (3.46);
		\begin{scope}
			\clip (0,-2.828) circle (3.46);
			\path[fill=black!20,even odd rule] (.816,-1.414) circle (2.582) (1.47,-.283) circle (2.498);
			\draw[wall]  (1.633,0) circle (2.582);
			\draw[wall]  (1.47,-.283) circle (2.498);
			\draw[wall]  (.816,-1.414) circle (2.582);
			%\draw[wall]  (1.4,-.404) circle (2.474);
			%\draw[wall]  (1.361,-.471) circle (2.465);
			%\draw[wall]  (1.336,-.514) circle (2.460);
			%\path[wall,fill=blue!50,even odd rule] (1.336,-.514) circle (2.460) (1.113,-.9) circle (2.460);
			%\draw[wall]  (1.113,-.9) circle (2.460);
			%\draw[wall]  (1.089,-.943) circle (2.465);
			%\draw[wall]  (1.050,-1.01) circle (2.474);
			%\draw[wall]  (.98,-1.131) circle (2.498);
%			\draw[wall, draw=gpurple]  (1.225,-.707) circle (2.45);
		\end{scope}
		\begin{scope}
			\clip (-2.45,1.414) circle (3.46);
			\path[fill=black!20,even odd rule] (-1.837,0.354) circle (2.739) (-.612,-1.768) circle (2.739);
			\draw[wall]  (-1.837,0.354) circle (2.739);
%			\draw[wall]  (-1.781,.257) circle (2.691);
%			\draw[wall]  (-1.774,.244) circle (2.684);
%			\path[fill=blue!50,even odd rule] (-1.774,.244) circle (2.684) (-.676,-1.658) circle (2.684);
%			\draw[wall]  (-.676,-1.658) circle (2.684);
%			\draw[wall]  (-.668,-1.671) circle (2.691);
			\draw[wall]  (-.612,-1.768) circle (2.739);
%			\path[fill=gpurple,even odd rule] (-1.633,0) circle (2.582) (-.816,-1.414) circle (2.582);
		\end{scope}
		\begin{scope}
			\clip (2.45,1.414) circle (3.46);
			\path[fill=black!20,even odd rule] (-.816,1.414) circle (2.582) (.490,1.414) circle (2.498);
			\draw[wall]  (-.816,1.414) circle (2.582);
			%\draw[wall]  (-.490,1.414) circle (2.498);
			%\draw[wall]  (-.350,1.414) circle (2.474);
			%\draw[wall]  (-.272,1.414) circle (2.465);
			%\draw[wall]  (-.223,1.414) circle (2.460);
			%\path[wall,fill=blue!50,even odd rule] (.223,1.414) circle (2.460) (-.223,1.414) circle (2.460);
			%\draw[wall]  (.223,1.414) circle (2.460);
			%\draw[wall]  (.272,1.414) circle (2.465);
			%\draw[wall]  (.350,1.414) circle (2.474);
			\draw[wall]  (.490,1.414) circle (2.498);
			\draw[wall]  (.816,1.414) circle (2.582);
%			\draw[wall, draw=gpurple]  (0,1.414) circle (2.45);
			%\begin{scope}
				%\clip (0,1.414) circle (2.45);
				%\path[fill=black!20,even odd rule] (0,1.414) circle (2.45) (0,-2.828) circle (3.46);
			%\end{scope}
		\end{scope}
%		\begin{scope}
%			\clip (-.612,-1.768) circle (2.739);
%			\draw[wall]  (-.432,-1.581) circle (2.586);
%%			\draw[wall]  (-1.781,.257) circle (2.691);
%%			\draw[wall]  (-1.774,.244) circle (2.684);
%%			\path[fill=blue!50,even odd rule] (-1.774,.244) circle (2.684) (-.676,-1.658) circle (2.684);
%%			\draw[wall]  (-.676,-1.658) circle (2.684);
%%			\draw[wall]  (-.668,-1.671) circle (2.691);
%%			\draw[wall]  (-.612,-1.768) circle (2.739);
%			\path[fill=gpurple,even odd rule] (-.377,-1.523) circle (2.542) (.7,-.404) circle (2.157);
%		\end{scope}
		\begin{scope}
			\clip (0,-2.828) circle (3.46);
			\clip (1.633,0) circle (2.582);
			\draw[wall] (.612,.354) circle (2.121);
		\end{scope}
%		\begin{scope}
%			\clip (2.45,1.414) circle (3.46);
%			\clip (0,1.414) circle (2.45);
%			\draw[wall] (0,0) circle (2);
%		\end{scope}
%		\begin{scope}
%			\clip (1.22,-.707) circle (2.45);
%			\clip (0,-2.828) circle (3.46);
%			\draw[wall] (0,0) circle (2);
%		\end{scope}
%		\begin{scope}
%			\clip (0,0) circle (2);
%			\begin{scope}
%				\clip (-.816,1.414) circle (2.582);
%				\draw[wall] (-.612,.354) circle (2.121);
%			\end{scope}
%			\begin{scope}
%				\clip (-.490,1.414) circle (2.498);
%				\draw[wall] (-.408,.707) circle (2.160);
%			\end{scope}
%			\begin{scope}
%				\clip (-.350,1.414) circle (2.474);
%				\draw[wall] (-.306,.884) circle (2.208);
%			\end{scope}
%		\end{scope}
%		\begin{scope}
%			\clip (0,-2.828) circle (3.46);
%			\begin{scope}
%				\clip (2.45,1.414) circle (3.46);
%				\draw[wall] (.612,.354) circle (2.121);
%			\end{scope}
%			\begin{scope}
%				\clip  (.816,1.414) circle (2.582);
%				\draw[wall] (.408,.707) circle (2.160);
%			\end{scope}
%			\begin{scope}
%				\clip (.490,1.414) circle (2.498);
%				\draw[wall] (.306,.884) circle (2.208);
%			\end{scope}
%		\end{scope}
%		\draw[green!50!black,-angle 90] (0,0) to (-.551,-.318) to (-.328,-1.324) to (0,-1.677) to (1.705,-1.969) to (3.042,-1.171) to (4.325,-1.498) to (4.899,-2.828);
		\draw[green!50!black,->] plot [smooth,tension=.7] coordinates {(0,0) (-.551,-.318) (-.328,-1.324) (.274,-1.741) (1.705,-1.969) (3.042,-1.171) (4.325,-1.498) (4.899,-2.828)};
		\node[green!50!black,above] at (4.899,-2.828) {$p$};
		%\draw[green!50!black] plot [smooth,tension=.5] coordinates {(0,0)  (-.551,-.318)  (-.328,-1.324)  (-.432,-.997)  (1.705,-1.969) (2.818,-.813) (5.449,3.146)};
	\end{scope}
%	\draw[red] (-.897,.518) to (-1.705,-1.969) to (-3.346,-1.932) to (-.983,.946);
%	\draw[red] (-.704,2.847) to (0,3.864) to (1.508,.29) to (.897,.518);
%	\draw[red] (0,-1.035) to (1.508,.29);
%	\draw[red] (-.816,1.414) circle (2.582);
	\draw[wall,fill=black!20] (0,-1.035) arc (-71.6:-25.8:2.582) arc (64.2:75:3.464) arc (56.6:97.8:2.739) arc (-172.2:-165:3.464) arc 
	(161.3:360-149.2:2.498) arc (360-149.2:360-41.3:2.498) arc (-11.6:-167.6:2.582) arc (-77.6:-105.0:3.464) arc (360-123.4:65.6:2.739) arc (155.6:135.0:3.464) arc (101.3:-42.4:2.498) arc (-42.4:-101.3:2.498);
		
    	\end{tikzpicture}
    \caption{Stereographic projection of $16$ chambers of the scattering diagram $\D(\Q_{2,2,3})$, and the path of a green-to-red sequence}
    \label{fig: scattering223}
    
\end{figure}
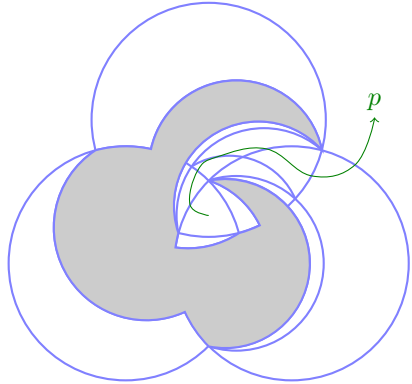

Figure \ref{fig: scattering223} draws a piece of $\D(\Q_{2,2,3})$; the grey denotes regions where the wall structure has not been computed.  Note the existence of a finite transverse path $p$ connecting the inner-most chamber to the exterior chamber; this corresponds to the green-to-red sequence in Figure \ref{fig: greentored}.  %A green-to-red sequence for $\Q_{2,2,3}$ corresponds to a finite transverse path in $\D(\Q_{2,2,3})$ which travels from the inner-most chamber to the exterior.  

%\end{ex}

\bibliography{MyNewBib}{}

\newcommand{\etalchar}[1]{$^{#1}$}
\def\cprime{$'$} \def\cprime{$'$} \def\cprime{$'$} \def\cprime{$'$}
  \def\cprime{$'$} \def\cprime{$'$}
\begin{thebibliography}{GHKK14}

\bibitem[ACC{\etalchar{+}}13]{ACCERV13}
Murad Alim, Sergio Cecotti, Clay C{{\'o}}rdova, Sam Espahbodi, Ashwin Rastogi,
  and Cumrun Vafa.
\newblock B{PS} quivers and spectra of complete {$N=2$} quantum field theories.
\newblock {\em Comm. Math. Phys.}, 323(3):1185--1227, 2013.

\bibitem[BBH11]{BBH11}
Andre Beineke, Thomas Br{\"u}stle, and Lutz Hille.
\newblock Cluster-cyclic quivers with three vertices and the {M}arkov equation.
\newblock {\em Algebr. Represent. Theory}, 14(1):97--112, 2011.
\newblock With an appendix by Otto Kerner.

\bibitem[BDP12]{BDP12}
Thomas Br{\"u}stle, Gr{\'e}goire Dupont, and Matthieu P{\'e}rotin.
\newblock On maximal green sequences.
\newblock 05 2012, 1205.2050.

\bibitem[Buc14]{Buc14}
Eric Bucher.
\newblock Maximal green sequences for cluster algebras associated to the
  n-torus.
\newblock 12 2014, 1412.3713.

\bibitem[DWZ10]{DWZ10}
Harm Derksen, Jerzy Weyman, and Andrei Zelevinsky.
\newblock Quivers with potentials and their representations {II}: applications
  to cluster algebras.
\newblock {\em J. Amer. Math. Soc.}, 23(3):749--790, 2010.

\bibitem[FZ02]{FZ02}
Sergey Fomin and Andrei Zelevinsky.
\newblock Cluster algebras. {I}. {F}oundations.
\newblock {\em J. Amer. Math. Soc.}, 15(2):497--529 (electronic), 2002.

\bibitem[FZ07]{FZ07}
Sergey Fomin and Andrei Zelevinsky.
\newblock Cluster algebras. {IV}. {C}oefficients.
\newblock {\em Compos. Math.}, 143(1):112--164, 2007.

\bibitem[GHK11]{GHK11}
Mark Gross, Paul Hacking, and Sean Keel.
\newblock Mirror symmetry for log {C}alabi-{Y}au surfaces {I}.
\newblock 06 2011, 1106.4977.

\bibitem[GHKK14]{GHKK}
Mark Gross, Paul Hacking, Sean Keel, and Maxim Kontsevich.
\newblock Canonical bases for cluster algebras.
\newblock {\em preprint, arxiv:1411.1394}, 11 2014, 1411.1394.

\bibitem[GMN10]{GMN10}
Davide Gaiotto, Gregory~W. Moore, and Andrew Neitzke.
\newblock Four-dimensional wall-crossing via three-dimensional field theory.
\newblock {\em Comm. Math. Phys.}, 299(1):163--224, 2010.

\bibitem[GPS10]{GPS10}
Mark Gross, Rahul Pandharipande, and Bernd Siebert.
\newblock The tropical vertex.
\newblock {\em Duke Math. J.}, 153(2):297--362, 2010.

\bibitem[Kel11]{Kel11c}
Bernhard Keller.
\newblock On cluster theory and quantum dilogarithm identities.
\newblock In {\em Representations of algebras and related topics}, EMS Ser.
  Congr. Rep., pages 85--116. Eur. Math. Soc., Z{\"u}rich, 2011.

\bibitem[KS11]{KS11}
Maxim Kontsevich and Yan Soibelman.
\newblock Cohomological {H}all algebra, exponential {H}odge structures and
  motivic {D}onaldson-{T}homas invariants.
\newblock {\em Commun. Number Theory Phys.}, 5(2):231--352, 2011.

\bibitem[KS13]{KS13}
Maxim Kontsevich and Yan Soibelman.
\newblock Wall-crossing structures in {D}onaldson-{T}homas invariants,
  integrable systems and {M}irror {S}ymmetry.
\newblock 03 2013, 1303.3253.

\bibitem[Pla11]{Pla11}
Pierre-Guy Plamondon.
\newblock Cluster characters for cluster categories with infinite-dimensional
  morphism spaces.
\newblock {\em Adv. Math.}, 227(1):1--39, 2011.

\bibitem[Rei10]{Rei10}
Markus Reineke.
\newblock Poisson automorphisms and quiver moduli.
\newblock {\em J. Inst. Math. Jussieu}, 9(3):653--667, 2010.

\bibitem[Sev14]{Sev14}
Ahmet~I. Seven.
\newblock Maximal green sequences of exceptional finite mutation type quivers.
\newblock {\em SIGMA 10 (2014), 089, 5 pages}, 06 2014, 1406.1072.

\bibitem[Yak]{YakMGS}
M.~Yakimov.
\newblock Maximal green sequences for double {B}ruhat cells.
\newblock In preparation.

\end{thebibliography}
\bibliographystyle{halpha}

\end{document}